\documentclass[12pt]{amsart}

\usepackage{amsfonts}
\usepackage{amsthm}
\usepackage{latexsym}
\usepackage{amsmath}
\usepackage{amssymb}
\usepackage{amscd}
\usepackage{amsmath}
\usepackage{mathrsfs}
\usepackage{graphicx}
\usepackage{a4wide}
\usepackage{color}
\usepackage{xcolor}
\usepackage{mathtools}
\usepackage{empheq}
\usepackage[font=small,labelfont=bf]{caption}
\definecolor{forestgreen}{RGB}{000,144,000}

\makeatletter
\CheckCommand*\refstepcounter[1]{\stepcounter{#1}%
      \protected@edef\@currentlabel
       {\csname p@#1\endcsname\csname the#1\endcsname}%
  }
\renewcommand*\refstepcounter[1]{\stepcounter{#1}%
    \protected@edef\@currentlabel
      {\csname p@#1\expandafter\endcsname\csname the#1\endcsname}%
  }
\def\labelformat#1{\expandafter\def\csname p@#1\endcsname##1}
\DeclareRobustCommand\Ref[1]{\protected@edef\@tempa{\ref{#1}}%
     \expandafter\MakeUppercase\@tempa
  }
\makeatother

\makeatletter
  \newcommand{\numberlike}[2]{%
     \expandafter\def\csname c@#1\endcsname{%
         \expandafter\csname c@#2\endcsname}%
  }
\makeatother

\def\DefaultNumberTheoremWithin{section}

\theoremstyle{plain}
\newtheorem{lemma}{Lemma}
     \numberwithin{lemma}{\DefaultNumberTheoremWithin}
     \labelformat{lemma}{Lemma~#1}

     \numberwithin{claim}{\DefaultNumberTheoremWithin}
     \numberlike{claim}{lemma}
     \labelformat{claim}{Claim~#1}
\newtheorem{theorem}{Theorem}
     \numberwithin{theorem}{\DefaultNumberTheoremWithin}
     \numberlike{theorem}{lemma}
     \labelformat{theorem}{Theorem~#1}
\newtheorem{corollary}{Corollary}
     \numberwithin{corollary}{\DefaultNumberTheoremWithin}
     \numberlike{corollary}{lemma}
     \labelformat{corollary}{Corollary~#1}
\newtheorem{proposition}{Proposition}
     \numberwithin{proposition}{\DefaultNumberTheoremWithin}
     \numberlike{proposition}{lemma}
     \labelformat{proposition}{Proposition~#1}

     \numberwithin{conjecture}{\DefaultNumberTheoremWithin}
     \numberlike{conjecture}{lemma}
     \labelformat{conjecture}{Conjecture~#1}

\theoremstyle{definition}
\newtheorem{definition}{Definition}
     \numberwithin{definition}{\DefaultNumberTheoremWithin}
     \numberlike{definition}{lemma}
     \labelformat{definition}{Definition~#1}

\theoremstyle{definition}

     \numberwithin{question}{\DefaultNumberTheoremWithin}
     \numberlike{question}{lemma}
     \labelformat{question}{Question~#1}

\theoremstyle{definition}
\newtheorem{problem}{Problem}
     \numberwithin{problem}{\DefaultNumberTheoremWithin}
     \numberlike{problem}{lemma}
     \labelformat{problem}{Problem~#1}

\theoremstyle{remark}
\newtheorem{remark}{Remark}
     \numberwithin{remark}{\DefaultNumberTheoremWithin}
     \numberlike{remark}{lemma}
     \labelformat{remark}{Remark~#1}
\theoremstyle{remark}

\newtheorem{example}{Example}
     \numberwithin{example}{\DefaultNumberTheoremWithin}
     \numberlike{example}{lemma}
     \labelformat{example}{Example~#1}
     
     \newtheorem{notation}{Notation}
     \numberwithin{notation}{\DefaultNumberTheoremWithin}
     \numberlike{notation}{lemma}
     \labelformat{notation}{Notation~#1}

     \labelformat{case}{Case~#1}
     \numberwithin{case}{lemma}

     \labelformat{step}{Step~#1}
     \numberwithin{step}{lemma}
     
\theoremstyle{plain}
\newtheorem{THEO}{Theorem}
     
     \labelformat{THEO}{Theorem~#1}

\newtheorem{PROP}{Proposition}
     \numberlike{PROP}{THEO}
     \labelformat{PROP}{Proposition~#1}

\newtheorem{COR}{Corollary}
     \numberlike{COR}{THEO}
     \labelformat{Cor}{Corollary~#1}

     \numberlike{LEM}{THEO}
     \labelformat{LEM}{Lemma~#1}

\numberwithin{equation}{section}

\labelformat{equation}{(#1)}
\labelformat{figure}{Figure~#1}
\labelformat{chapter}{Chapter~#1}
\labelformat{appendix}{Appendix~#1}
\labelformat{section}{\S~#1}
\labelformat{subsection}{\S~#1}
\labelformat{subsubsection}{\S~#1}

\def\R{\mathbb R}
\def\Z{\mathbb Z}

\def\d{\partial}

\def\g{\gamma}
\def\b{\beta}
\def\a{\alpha}

\def\cP{{\mathcal P}}

\def\sI{{\mathsf I}}
\def\sM{{\mathsf M}}
\def\sR{{\mathsf R}}

\def\bfc{{\mathbf c}}
\def\sfy{{\mathsf{y}}}
\def\sfx{{\mathsf{x}}}
\def\crit{{\mathsf{crit}}}
\def\rank{{\mathsf{rank}}}
\def\bfOm{{\Omega}}
\def\paOm#1#2{{\Omega}_{#1,\,|\sim|' #2}}

\def\om{\omega}

\def\Gr{{\mathfrak G}}

\begin{document}

\title[Real polynomials with constrained real divisors. I.  Fundamental groups] {Real polynomials with constrained real divisors. I.  Fundamental groups}    

\author[G.~Katz]{Gabriel Katz}
\address{MIT, Department of Mathematics, 77 Massachusetts Ave., Cambridge, MA 02139, U.S.A.}
\email {gabkatz@gmail.com }

\author[B.~Shapiro]{Boris Shapiro}
\address{Stockholm University, Department of Mathematics, SE-106 91
Stockholm, Sweden}
\email {shapiro@math.su.se }

\author[V.~Welker]{Volkmar Welker}
\address{Philipps-Universit\"at Marburg, Fachbereich Mathematik und Informatik, 35032 Marburg, Germany}
\email {welker@mathematik.uni-marburg.de}

\begin{abstract} 
  In the late 80s, V.~Arnold and V.~Vassiliev initiated the topological study 
  of the space of real univariate polynomials of a given 
  degree $d$ and with no real roots of multiplicity exceeding a given positive 
  integer. Expanding their studies, we consider the spaces
  $\cP^{\bfc \Theta}_d$ of real monic univariate  polynomials of degree 
  $d$ whose real divisors avoid sequences of root multiplicities, taken 
  from a given poset $\Theta$ of compositions which is closed under certain 
  natural combinatorial operations. 
	
  In this paper, we concentrate on the 
  fundamental group of $\cP^{\bfc \Theta}_d$ and of some related topological 
  spaces. We find explicit presentations for the groups 
	$\pi_1(\cP^{\bfc \Theta}_d)$ in terms of generators and relations and show that in a number of cases they are free 
  with rank bounded from above by a quadratic function in $d$. We also show
  that $\pi_1(\cP^{\bfc \Theta}_d)$ stabilizes for $d$ large.

  The mechanism 
  that generates $\pi_1(\cP^{\bfc \Theta}_d)$ has similarities with the
  presentation of the braid group as the fundamental group of the space 
  of complex monic degree $d$ polynomials with no multiple roots and with 
  the presentation of the fundamental group of certain ordered configuration
  spaces over the reals which appear in the work of Khovanov.

  We further show that the groups $\pi_1(\cP^{\bfc \Theta}_d)$ admit an 
  interpretation as special bordisms of immersions of $1$-manifolds into the 
  cylinder $\R \times S^1$, whose images avoid the tangency patterns from 
  $\Theta$ with respect to the generators of the cylinder. 
\end{abstract}

\date{\today}

\maketitle

\setcounter{page}{1}

\section{Introduction} \label{sec:intro}

\subsection{Motivation and Outline of Results}

In \cite{Ar}, V.~Arnold proved the following Theorems A--D, which were later 
generalized by V.~Vassiliev, see \cite{Va}.  These results are the main source 
of motivation and inspiration for our study. In the formulations of these 
theorems, we keep the original notation of \cite{Ar}, which we will abandon 
later on. In what follows, theorems, conjectures, etc., labeled by letters, 
are borrowed from the existing literature, while those labeled by numbers 
are hopefully new. 

\begin{THEO} \label{th:Arnold1} 
  The fundamental group of the space of smooth functions 
  $f : S^1 \rightarrow \R$ without critical points of multiplicity higher 
  than $2$ on a circle $S^1$ is isomorphic to the group of integers $\Z$. 
  
  The space of smooth functions $f : \R \rightarrow \R$ without critical 
  points of multiplicity higher than $2$ and which, for arguments $|x|>1$, 
  coincide either with $x$ or with $x^2$ also have the fundamental group $\Z$. 
\end{THEO}

\begin{THEO} \label{th:Arnold2} 
  The latter fundamental group is naturally isomorphic to the group 
  $\mathcal B$ of $A_3$-cobordism classes of embedded closed plane curves 
  without vertical \footnote{In our convention,  the curves in the $tx$-plane 
  do not have inflections with respect to the coordinate line $\{t=const\}$.} 
  tangential inflections.  
\end{THEO}

The generator of $\mathcal B$ is shown as the kidney-shaped loop in 
~\ref{fig:kidney}(a).

\begin{remark} 
  The multiplication of the cobordism classes in $\mathcal B$  is defined as  
  the disjoint union of curves, embedded in the half-planes 
  $\{(t,x)|\, t < 0\}$ and $\{(t,x)|\, t > 0\}$, 
  and the inversion is the change of sign of $t$.  
\end{remark}
  
For $1 \leq k \leq d$, let $G_k^d$ be the space  of 
real monic polynomials $x^d + a_{d-1}x^{d-1} + \cdots + a_0 \in \R[x]$
with no real roots of multiplicity greater than $k$.

\begin{THEO}\label{th:Arnold3}  
  If $k < d < 2k+1$, then $G_k^d$ is diffeomorphic to the product of a 
  sphere $S^{k-1}$ by an Euclidean space. In particular,  for all $i$ and $k < d < 2k+1$,
  $$\pi_i(G_k^d)\simeq \pi_i(S^{k-1}) $$
\end{THEO}

An analogous result holds for the space of polynomials whose sum of roots 
vanishes, i.e., polynomials with the vanishing coefficient $a_{d-1}$.
	
\begin{THEO}\label{th:Arnold4}
  The homology groups 
  with integer coefficients of the space $G_k^d$ 
  are nonzero only for dimensions which are the multiples of $k - 1$ and 
  less or equal to $d$. For $(k-1)r \leq d$, we have 
  $$H_{r(k-1)}(G_k^d) \simeq \Z.$$
\end{THEO}

The main goal of this paper and its sequel \cite{KSW} is to generalize 
Theorems A--D to the situation where the multiplicities of the real roots 
\emph{avoid a given set of patterns} $\Theta$. In our more general situation, 
the fundamental group of such polynomial spaces can be non-trivial and deserves 
a separate study, which is carried out below. We will see that the mechanism 
by which these fundamental groups are generated is in some sense similar to the one 
that produces the braid groups as the fundamental groups of spaces of complex 
degree $d$ monic polynomials with no multiple roots. The mechanism is also 
related to Khovanov's paper \cite{Kh}, which studies the topology of spaces $K_d$, obtained from the space $\R^d$ by 
removing vectors $(x_1, \ldots, x_d)$ such that $x_j = x_k = x_l$ for some distinct $j, k, l$, or/and 
vectors of the form $x_j = x_k$ and  $x_l = x_m$ for some distinct $j, k, l, m$.  

This space is an ordered analog of the space of real 
polynomials containing the polynomials with only real roots, no roots of 
multiplicity $3$ or higher, and no two different roots both of multiplicity $2$ 
or higher. \ref{figure0} shows two contrasting  images of 
objects from \cite{Kh} and from the present paper, each representing a \emph{loop} 
in the relevant configuration space.  

\begin{figure}

\begin{center}
\includegraphics[width=0.8\textwidth]{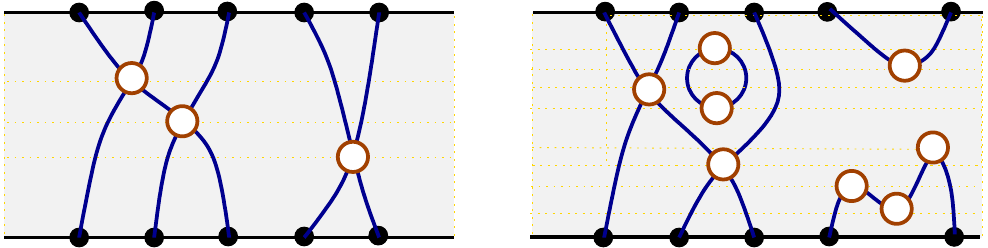}%
\end{center}
~\vskip-5.6cm
\setlength{\unitlength}{3947sp}%
\begingroup\makeatletter\ifx\SetFigFont\undefined%
\gdef\SetFigFont#1#2#3#4#5{%
  \reset@font\fontsize{#1}{#2pt}%
  \fontfamily{#3}\fontseries{#4}\fontshape{#5}%
  \selectfont}%
\fi\endgroup%
\begin{picture}(7962,2829)(1168,-3001)
\put(3200,-2986){\makebox(0,0)[lb]{\smash{{\SetFigFont{12}{14.4}{\rmdefault}{\mddefault}{\updefault}{\color[rgb]{0,0,0}(a)}%
}}}}
\put(6500,-2986){\makebox(0,0)[lb]{\smash{{\SetFigFont{12}{14.4}{\rmdefault}{\mddefault}{\updefault}{\color[rgb]{0,0,0}(b)}%
}}}}
\end{picture}%
\caption{(a) A graphic representation of a loop in the Khovanov's space $K_5$ and (b) in  the space $\mathcal P_7^{\mathbf c \Theta}$ of real degree $7$ polynomials with no real roots of multiplicity $\geq 3$ and no pairs of real roots of multiplicity $\geq 2$. 
} 
\label{figure0}
\end{figure}

\smallskip

Besides the studies of V.~Arnold \cite{Ar} and V.~Vassiliev \cite{Va}, 
the second major motivation for this paper comes from results of the 
first author connecting the cohomology of spaces of real polynomials that avoid 
certain patterns of root multiplicities with  characteristic classes, arising in the theory of traversing flows, see \cite{Ka}, \cite{Ka1}, and \cite{Ka3}. For traversing vector flows on 
compact manifolds $X$ with boundary $\d X$ and with a priori forbidden 
tangency patterns of their trajectories to $\d X$, the spaces  
of polynomials avoiding the same patterns play a fundamental role which is similar to the one played by
Gra\ss mannians in the category of vector bundles.

Let $\cP_d$ denote the space of real monic univariate polynomials of degree 
$d$. Given a polynomial $P(x) = x^d + a_{d-1}x^{d-1} + \cdots + a_0 \in \cP_d$, 
we define its \emph{real divisor} $D_\R(P)$ as the multiset 
$$\underbrace{x_1 = \cdots = x_{i_1}}_{\omega_1 = i_1} < 
\underbrace{x_{i_1+1} = \cdots = x_{i_1+i_2}}_{\omega_2 = i_2-i_1} <
\cdots < \underbrace{x_{i_{\ell-1} +1} = \cdots = x_{i_\ell}}_{\omega_\ell
= i_\ell-i_{\ell-1}}$$
of the real roots of $P(x)$. The tuple 
$\om = (\omega_1 ,\ldots, \omega_\ell)$ is called 
the (ordered) \emph{real root multiplicity pattern} of $P(x)$. 
Let $\mathring{\sR}^\omega_d$ be  the set of all polynomials with root multiplicity pattern $\omega$, and let 
$\sR_d^\omega$ be the closure of $\mathring{\sR}^\omega_d$ in $\cP_d$. 

For a given collection $\Theta$ of root multiplicity patterns, we 
consider the union $\cP_d^\Theta$ of the subspaces $\mathring{\sR}_d^{\omega}$
over all $\omega \in \Theta$. We denote by $\cP_d^{\bfc\Theta} := \cP_d \setminus \cP^\Theta_d$ its complement. We restrict our studies to the case when 
$\cP_d^\Theta$ is closed in $\cP_d$ and call such $\Theta$ \emph{closed}.

We observe in \ref{lm2} that, for any $\Theta$ containing the minimal element $(d)$ (and thus for any closed $\Theta$), the space
$\cP_d^\Theta$ is contractible. Thus it makes more sense to consider 
its one-point compactification $\bar{\cP}_d^\Theta$. For a closed $\Theta$, 
the latter is the 
union of the one-point compactifications $\bar{\sR}_d^\omega$ of the 
$\sR_d^\omega$ for $\omega \in \Theta$ with the  identified points at infinity. 
By Alexander duality in $\bar{\cP}_d \cong S^{d}$, 
we get the relation  
$$\bar H^j(\cP_d^{\bfc\Theta}; \Z) \approx \bar H_{d - j -1}(\bar{\cP}_d^{\Theta}; \Z)$$
in reduced (co)homology.  
This  implies that the spaces $\cP_d^{\bfc\Theta}$ and $\bar{\cP}_d^{\Theta}$ 
carry equivalent (co)homological information.

\begin{example}
  For $\Theta$ comprising all $\om$'s with at least one component greater
  than or equal to $k$, we have that $\Theta$ is closed and
  $\cP_d^{\bfc\Theta} \cong G_k^d$
  (see  \ref{th:Arnold3} and  \ref{th:Arnold4}).
\end{example}

In this paper and its sequel \cite{KSW}, for a closed $\Theta$, we aim at describing the topology of
$\bar{\cP}^\Theta$ and $\cP^{\bfc \Theta}$
in terms of combinatorial properties of $\Theta$.

\medskip 

\noindent 
\emph{Outline of the results:} 
In \ref{sec:fundamental} we prove our main results, namely, \ref{lem2.14}, \ref{main} and \ref{corasym}, generalizing
\ref{th:Arnold3} in the case $k = 2$. \ref{lem2.14} gives a presentation of 
$\pi_1(\cP_d^{\bfc\Theta})$ by generators and relations while \ref{main} states that
under certain conditions on $\Theta$,  
$\pi_1(\cP_d^{\bfc\Theta})$ is a \emph{free group}. In \ref{corasym} we show that $\pi_1(\cP_d^{\bfc\Theta})$ stabilizes for $d$ large. 
In \ref{cor:arnA}, we provide an analog of \ref{th:Arnold1} and show that the 
fundamental group of the space of real 
monic polynomials of a fixed odd degree $d > 1$ with no real critical points 
of multiplicity higher than $2$ is isomorphic to $\Z$. 

Additionally, when $\Theta$ consists of all $\om$'s for which 
$\sR_d^\omega$ has codimension $\geq 2$,
we show in \ref{lm2b}, without using the results from \cite{KSW}, that
the number of generators of the free group $\pi_1(\cP_d^{\bfc\Theta})$ is 
equal to $\frac{d(d-2)}{4}$ for an even $d$, and to $\frac{(d-1)^2}{4}$ for 
an odd $d$. Moreover, we show that in this case 
$\cP_d^{\bfc\Theta}$ is homotopy equivalent to a wedge of circles and hence is 
a $K(\pi, 1)$-space.

In \ref{sec:cobordism} a special cobordism theory which allows us to formulate and 
prove an analog of \ref{th:Arnold2} is developed.  \ref{cobordism} claims 
that the free group $\pi_1(\cP^{\bfc \Theta}_d)$, where $\Theta$ consists of all $\om$'s for which $\sR_d^\omega$ has codimension $\geq 2$, 
admits an 
interpretation as special \emph{bordisms of immersions} of $1$-manifolds into 
the cylinder $\R \times S^1$, i.e., immersions whose images avoid the tangency 
patterns from $\Theta$ with respect to the generators of the cylinder.


\subsection{Cell structure on the space of real univariate polynomials}
\label{sec:cellstructure}

First we recall a well-known stratification of the space of real univariate polynomials of a given degree.  

For any  real polynomial $P(x)$, we have already defined 
its real divisor $D_\R (P)$, i.e. the ordered set of its real zeros, counted with their multiplicities.

We have also associated to a polynomial $P(x) \in \R[x]$ and $D_\R(P)$ its 
real root multiplicity pattern $\om_P := (\omega_1,\ldots, \omega_\ell)$.
The combinatorics of multiplicity patterns will play the key role in our
study. Let us fix some terminology and notation for multiplicity
patterns.

\begin{definition}
 A sequence $\om=(\om_1,\dots, \om_\ell)$ of positive 
 integers is called a {\it composition} of the number $|\om| := \om_1 + \cdots 
 +\om_\ell$. We call $|\om|$ the \emph{norm} of $\om$. (We  allow the \emph{empty composition} $\om = ()$, whose norm $ |()| = 0$). 
 
	We call the number $|\om|' :=  (\om_1 -1) + \cdots 
 +(\om_\ell -1)$ the \emph{reduced norm} of $\om$.
 \end{definition}

Evidently, for a given composition $\om$, the stratum $\sR^\omega_d$ is empty 
if and only if either $|\om | > d$, or $|\om | \leq d$ and 
$|\om | \not \equiv d \mod 2$. 

\smallskip

\begin{notation}
  We denote by $\bfOm$ the set of all compositions of natural numbers.  
  For a given positive integer $d$,  we write $\bfOm_{\langle d]}$ for the 
  set of all $\om \in \bfOm$, such that $|\om| \leq d$ and $|\om|
  \equiv d \mod 2$. 
  We denote by $\paOm{\langle d]}{\geq \ell}$ the subset of 
  $\bfOm_{\langle d]}$, 
  consisting of all compositions $\om \in \bfOm_{\langle d]}$ 
  with $|\om|' \geq \ell$. 
  Analogously, we define $\paOm{\langle d]}{= \ell}$ as the subset of 
  $\bfOm_{\langle d]}$, 
  consisting of all compositions $\om \in \bfOm_{\langle d]}$ 
  with $|\om|' = \ell$. 
\end{notation}

Now we define two (sequences of) operations on $\bfOm$ 
that will govern our subsequent considerations, see also \cite{Ka}.

\smallskip

The \emph{merge operation} 
$\sM_j: \bfOm \to \bfOm$ \;  
sends $\om = (\om_1,\ldots , \om_\ell)$ to the composition $$\sM_j(\om) = (M_j(\om)_1,
\ldots, M_j(\om)_{\ell-1}),$$ 
where, for any $j \geq \ell$, one has $\sM_j(\om) = \om$,  and for $1 \leq j < \ell$, one has 
\begin{eqnarray}\label{eq2.1a} 
\sM_j(\omega)_i & = & \omega_i \; \textrm{ if }\, i < j,\\ \nonumber
\sM_j(\omega)_j  & = & \omega_j + \omega_{j+1},\\ \nonumber
\sM_j(\omega)_i  & = & \omega_{i + 1} \; \textrm{ if }\, i+1 < j \leq  \ell-1.
\end{eqnarray}

\medskip
Similarly, we define the \emph{insertion operation} 
$\sI_j: \bfOm \to \bfOm$ that 
sends $\om = (\om_1,\ldots , \om_\ell)$ to the composition $\sI_j(\om) = (I_j(\om)_1,
\ldots, I_j(\om)_{\ell+1}),$ where for any $j > \ell+1$,  one has $\sI_j(\om) = \om$, and  for $1 \leq j \leq \ell+1$, one has 
\begin{eqnarray}\label{eq2.2a}  
\sI_j(\omega)_i & = & \omega_i\; \textrm{ if }\, i < j, \\ \nonumber
\sI_j(\omega)_j & = & 2,\\ \nonumber
\sI_j(\omega)_i & = & \omega_{i - 1} \; \textrm{ if }\, j \leq i \leq \ell+1. 
\end{eqnarray}

\medskip 
The next proposition collects some basic properties of $\sR^\omega_d$, see 
\cite[Theorem 4.1]{Ka} for details.

\begin{PROP}\label{lm1}
  Take $d \geq 1$ and $\om = (\om_1,\ldots, \om_\ell) \in \bfOm_{\langle d]}$. 
  Then $\mathring{\sR}^\omega_d\subset \cP_d$ is an (open) cell of 
  codimension $|\om|'$. 
  Moreover, 
	$\sR^\omega_d$ is the union of the cells 
	$\{\mathring{\sR}_d^{\omega'}\}_{\omega'}$, 
  taken over all $\omega'$ that are obtained from 
  $\omega$ by a sequence of merging and insertion operations.
  In particular,
  \begin{itemize}
	  \item[(a)] the cell $\mathring{\sR}^\om_d$ has (maximal) dimension $d$ if and only
       if $\om=(\underbrace{1,1,\dots, 1}_{\ell})$ for 
       $0 \leq \ell \leq d$ and $\ell \equiv d \mod 2$, \smallskip
       
     \item[(b)] the cell $\mathring{\sR}^\om_d$ has dimension $1$ if
	     and only if $\om=(d)$. In this case, $\mathring{\sR}^{(d)} = 
	     \sR^{(d)} = \{ (x-a)^d~|~a \in \R\}$. 
  \end{itemize}
\end{PROP}

	Geometrically speaking, if a typical point in $\mathring{\sR}^\omega_d$ approaches the boundary $\d\sR^\omega_d := \sR_d^\om \setminus \mathring{\sR}^\omega_d$, 
then either there is at least one value of the index $j$ such that the distance between 
the $j$\textsuperscript{th} and $(j+1)$\textsuperscript{st} largest root in 
$D_\R(P)$ goes to $0$, or there are two complex-conjugate real roots that converge 
to a double real root, which  then either is $j$\textsuperscript{th} largest
or adds $2$ to the multiplicity of the $j$\textsuperscript{th} largest real root.

	The first situation corresponds to the application of the merge operation 
$\sM_j$ to $\omega$, and the second  either to the application of the insertion $\sI_j$. Of course, there exist points in the boundary of $\d \sR^\omega_d$ that can be reached from $\mathring{\sR}^\omega_d$ by applying  sequentially a number of inserts and merges. \smallskip
 
Note that $|\sim|$ is preserved under the merge operations, while 
the insert operations increase $|\sim|$ by $2$ and thus preserve its parity.
  
By \ref{lm1}, the merge and the insert operations can be used to define a 
\emph{partial order} ``$\succ$''  on the set $\bfOm$ of all 
compositions, reflecting the adjacency of the non-empty open cells 
$\mathring{\sR}_d^\omega$.

\begin{definition}\label{def2.1} 
  For $\omega, \omega' \in \bfOm$, we say that $\omega'$ 
  {\it is smaller than} $\omega$ (notation  ``$\,\omega \succ \omega'\,$'' or
   ``$\,\omega' \prec \omega\,$''),
  if $\omega'$ can be obtained from $\omega$ by a sequence of merge  and 
  insert operations $\{\sM_j\}$, $j \geq 1$, and  
  $\{\sI_j\}$, $j \geq 0$. 
\end{definition}

From now on, we will consider a subset $\Theta \subseteq \bfOm$ as a 
\emph{poset}, ordered by $\succ$.
As an immediate consequence of \ref{lm1}, we get the
following statement.

\begin{COR} 
  \label{cor:cw-structure}
  For $\Theta \subseteq \bfOm_{\langle d]}$, 
  \begin{itemize}
    \item[(i)] $\cP^\Theta_d$ is closed in $\cP_d$ if and only if,
	    for any $\omega \in \Theta$ and $\om' \in \bfOm_{\langle d]}$,
		  the relation $\omega' \prec \omega$ implies 
                   $\omega' \in \Theta$;
     
    \item[(ii)] if $\cP^\Theta_d$ is closed in $\cP_d$, then the one-point compactification
	    $\bar{\cP}^\Theta_d$ carries the structure of a compact CW-complex 
		  with open cells $\{\mathring{\sR}_d^\omega\}_{\omega \in \Theta}$, labeled by
		  $\omega \in \Theta$, and the unique $0$-cell, represented by the  point $\bullet$ at infinity.
  \end{itemize}
\end{COR}

Recall that we call $\Theta \subseteq \bfOm_{\langle d]}$ closed 
if $\cP^\Theta_d$ is closed in $\cP_d$. Hence 
\ref{cor:cw-structure} has the following immediate reformulation.

\smallskip

\begin{center} 
  $\Theta \subseteq \bfOm_{\langle d]} \text{ is closed }$
	
	$\Updownarrow$
 
	$\text{for any } 
  \omega \in \Theta \text{ and } \om' \in \bfOm_{\langle d]},
  \text{ the relation } \om' \prec \om \text{ implies } 
  \om' \in \Theta.$ 
\end{center}

Now, we are in position to give a precise formulation of the main questions discussed in 
 this paper and its sequel \cite{KSW}: 

\begin{problem}\label{prob:main} 
  For a given closed poset $\Theta \subseteq 
  \bfOm_{\langle d]}$, 
  
  \begin{itemize}
  
	\item[$\triangleright$] calculate the homotopy groups 
  $\pi_i(\bar{\cP}_d^{\Theta})$  and $\pi_i({\cP}_d^{\bfc\Theta})$ in terms of the combinatorics
  of $\Theta$;\footnote{In full generality, this goal is as illusive as computing the homotopy groups of spheres; in fact,  as testified by \ref{th:Arnold3}, for special $\Theta$'s,  the two problems are intimately linked.} 
	
	\item[$\triangleright$]  calculate the integer homology of 
  $\bar{\cP}_d^{\Theta}$ or, equivalently, the integer 
  cohomology of $\cP_d^{\bfc\Theta}$ in terms of the combinatorics
  of $\Theta$. 
  
\end{itemize}
	
\end{problem}

In this paper we concentrate on the fundamental groups of 
$\bar{\cP}_d^{\Theta}$ and $\cP^{\bfc\Theta}_d$.
Questions about the (co)homology of $\bar{\cP}_d^{\Theta}$ and 
$\cP^{\bfc\Theta}_d$ will be addressed in \cite{KSW}.

\medskip
\noindent
\emph{Acknowledgements.} 
 The second author wants to acknowledge the financial support of his research by the Swedish Research council through the grant 2016-04416. 
 The third author wants to thank department of Mathematics of Stockholm 
University for its hospitality in May 2018.
He also was partially supported by an NSF grant DMS 0932078, administered by 
the Mathematical  Sciences  Research  Institute  while  the  author  was  in  
residence  at  MSRI during the complimentary program 2018/19.
During both visits substantial progress on the project was made.

We also would like to thank Alex Suciu for enlightening discussions about our fundamental groups and for pointing out a mistake in earlier version of \ref{extorsion}.

\section{Computing $\pi_1(\bar{\cP}_d^{\Theta})$ and $\pi_1(\cP_d^{\bfc\Theta})$}
\label{sec:fundamental}

\subsection{Homotopy type of $\cP_d^\Theta$ and the fundamental group $\pi_1(\bar{\cP}_d^{\Theta})$}

The following simple statement gives us a start on the homotopy type of the 
polynomial spaces under consideration.
In the proof, we use the map $q : \cP_d \times [0,+\infty) \rightarrow \cP_d$ which sends 
each pair $(P(x),\lambda)$, where 
$P(x) = x^d + a_{d-1} x^{d-1} + \cdots 
+ a_0$ and $\lambda \in [0,+\infty)$, to the polynomial $x^d + a_{d-1} \, \lambda \,x^{d-1} + \cdots + a_0 \, \lambda^{d}$.
Hence, this transformation amounts to the multiplication of all roots of $P(x)$ by $\lambda$. 

\begin{lemma}\label{lm2} 
  For any poset $\Theta \subseteq \bfOm_{\langle d]}$ that  
  contains $(d)$, the space  
  $\cP^{\Theta}_d\subseteq \cP_d$ is contractible.   
  In particular, for any closed poset $\Theta  \subseteq 
  \bfOm_{\langle d]}$, the space $\cP^{\Theta}_d$ is 
  contractible.
\end{lemma}

\begin{proof} 
  For $P(x) \in \cP^\Theta_d$ and $\lambda \geq 0$, the roots of
  $q(P(x),\lambda)$ are the roots of $P(x)$, being multiplied by $\lambda$.  
  In particular, we have $q(P(x),0) = x^d \in \cP_d^{(d)}$.
  Thus by $(d) \in \Theta$ we have $q(P(x),\lambda) \in \cP_d^\Theta$ for 
  $P(x) \in \cP_d^\Theta$. 
  Then the restriction of $q$ to $[0,1] \times \cP^{\Theta}_d$ is a 
  well-defined homotopy between the identity map and the
  constant map that sends $\cP^{\Theta}_d$ to $x^d$.
  The assertion now follows.
\end{proof}

In contrast to $\cP_d^{\Theta}$, its one-point compactification 
$\bar{\cP}_d^{\Theta}$ often has non-trivial topology for closed posets 
$\Theta$. A simple example of such a situation is 
$\bar{\cP}_d^{\bfOm_{\langle d]}} = \bar{\cP}_d \cong S^d$. Other 
examples, including the case treated in \ref{th:Arnold3} and \ref{th:Arnold4}, 
show that $\cP_d^{\bfc\Theta}$ can have non-trivial topology as well.

Besides the map $q$, the following map $p$ has been frequently used in the 
literature on the topology of spaces of univariate polynomials. 
The map $p : \cP_d \rightarrow \cP_d$ sends 
$P(x) = x^d + a_{d-1}x^{d-1} + \cdots + a_0 \in \R[x]$ to 
$P(x-\frac{a_{d-1}}{d})$. The map $p$ preserves the stratification 
$\{\mathring{\sR}^\omega_d\}_\om$ and is a fibration with the fiber $\R$. 
Thus, for a closed poset $\Theta \subseteq \bfOm_{\langle d]}$, the 
restriction $p|_{\cP^\Theta_d} : \cP^\Theta_d \rightarrow \cP^\Theta_d$  
is also a fibration with fiber $\R$. Its image $\cP^\Theta_{d,0}$ consists of 
all polynomials in $\cP^\Theta_{d}$ with the vanishing coefficient at $x^{d-1}$, 
i.e. with vanishing root sum.  Therefore, we get a homeomorphism  
$\bar{\cP}^\Theta_d \cong \tilde\Sigma \bar{\cP}^\Theta_{d,0}$. 
Here $\tilde\Sigma\,X$ denotes the reduced  suspension of a space $X$.

Since $q(P(x),\lambda)$ amounts to multiplying the roots of $P(x)$ by
$\lambda$, their sum is multiplied by $\lambda$ as well. In particular, $q$ preserves 
$\cP^\Theta_{d,0}$. 
Using the map $$a_0 + \cdots +a_{d-1}x^{d-1}+x^d \, \mapsto \, (a_0,\ldots, a_{d-1})$$ we 
identify $\cP_d$ with Euclidean $d$-space with respect to its metric structure.
Clearly, for $\lambda \geq 0$ and fixed $0 \neq P(x) \in \cP_d$, the norm
$\|q(P(x),\lambda)\|$ is strictly monotone in $\lambda$. 
Since $\|q(P(x),0)\|=0$, there is a unique $\lambda_P > 0$ for which 
$\|q(P(x),\lambda_P)\|=1$. 
Let $S^{d-1} \subset \cP_d$ be the unit sphere.
Therefore, 
on $\cP^\Theta_{d,0} \setminus \{0\}$ the map $P(x) \mapsto q(P(x),\lambda_P)$ 
is a deformation retraction to the closed subspace  
$S^{d-1} \cap \cP^{\Theta}_{d,0}$ of $S^{d-1}$. Thus, we get a homeomorphism 
$\bar{\cP}_{d,0}^\Theta \cong \Sigma (S^{d-1} \cap \cP^{\Theta}_{d,0})$.
Note, that  we consider $\Sigma\,\emptyset$ as the discrete two-point 
space. This analysis implies the following claim. 

\begin{theorem} 
  \label{pro:pitheta}
  For  any closed poset  $\Theta \subseteq \bfOm_{\langle d]}$, we get 
   $\pi_{1}(\bar{\cP}_d^\Theta) = 0$, unless $\Theta = \{ (d) \}$. 
  If $\Theta = \{(d)\}$ then $\bar{\cP}_d^\Theta \cong S^1$. 
\end{theorem} 
\begin{proof}
  By the arguments preceding the theorem, we have $\bar{\cP}_d^\Theta \cong
  \tilde\Sigma \Sigma (S^{d-1} \cap \cP^{\Theta}_{d,0})$. 
  Therefore, 
  $\bar{\cP}_d^\Theta$ is simply connected, unless 
  $S^{d-1} \cap P^{\Theta}_{d,0}$ is
  empty. But this can only happen if $\Theta = \{(d)\}$. It is easily seen that
	$\mathring{\sR}^{(d)}_d = \{ (x-\alpha)^d~|~\alpha\in \R\}
  \cong \R$. Hence its one-point compactification is $S^1$. 
  \end{proof}

 Note, that the argument employed in the proof of 
 \ref{pro:pitheta} also implies the following isomorphism in homology
 $$\bar H_i(S^{d-1} \cap P^{\Theta}_{d,0};\, \Z) \approx 
 \bar H_{i+2}(\bar{\cP}_d^\Theta;\, \Z)$$ for all $i \geq 0.$

\subsection{The fundamental group of the complement of the codimension two skeleton of the space $\mathcal P_d$}

In this section, we describe the fundamental group 
$\pi_1(\cP_d^{\bfc\paOm{\langle d]}{\geq 2}})$. Recall, that 
we collect in $\paOm{\langle d]}{\geq 2}$ all $\om \in \bfOm_{\langle d]}$ such 
that $|\om|' \geq 2$. Thus  $\cP_d^{\bfc\paOm{\langle d]}{\geq 2}}$ 
is the complement of the codimension $2$ skeleton of $\cP_d$ in our celluation.
\smallskip

The outline  of the arguments below is as follows. 
The space $\cP_d^{\bfc\Theta}$ is an open $d$-dimensional manifold without 
boundary, stratified into cells of dimensions $d$ and $d-1$. Moreover, for 
any pair of top-dimensional cells, there exists at most one $(d-1)$-dimensional cell which separates them. Any such manifold is homotopy equivalent to a graph 
whose vertices are top-dimensional cells and whose edges connect pairs of 
adjacent top-dimensional cells. \smallskip

We associate a graph $\Gr_d$ with the space 
$\cP_d^{\bfc\paOm{\langle d ]}{\geq 2}}$, subdivided into open $d$-cells by the $(d-1)$-cells. 
The set of vertices of $\Gr_d$ is the union of the sets
$\paOm{\langle d]}{=1}$ and $\paOm{\langle d]}{=0}$.
We connect vertices $\om \in \paOm{\langle d]}{=1}$ and 
$\om' \in \paOm{\langle d]}{=0}$  by an edge $\{\om,\om'\}$, 
if the $(d-1)$-cell $\mathring{\sR}_d^{\om}$ 
lies in the boundary of the closure 
$\sR_d^{\omega'}$ of $\mathring{\sR}_d^{\omega'}$.
In particular, the edges of $\Gr_d$ correspond to single insertion and merging 
operations, applied  to compositions from $\paOm{\langle d]}{=0}$.
As usual, we identify the graph $\Gr_d$ 
with the $1$-dimensional simplicial complex, defined by its vertices and 
edges (see \ref{fig3} for the example of $\Gr_6$).  

We embed the graph 
$\Gr_d$ in $\cP_d^{\mathbf {c}\paOm{\langle d ]}{\geq 2}}$ by mapping the 
vertex of $\Gr_d$, labeled by $\om \in \paOm{\langle d]}{=1}$, to a 
preferred point $w_\om$ in the $(d-1)$-cell $\mathring{\sR}_d^{\omega}$ and 
the vertex of $\Gr_d$, labeled by $\om \in \paOm{\langle d]}{=0}$, to a 
preferred point $e_\om$ in the $d$-cell $\mathring{\sR}_d^{\omega}$. 
Then we identify each edge $\{\om, \om'\}$ of $\Gr_d$, where
$\om \in \paOm{\langle d]}{=1}$, $\om \in \paOm{\langle d]}{=0}$,
with a smooth path $[w_\om, e_{\om'}]$ such that the semi-open segment 
$(w_\om, e_{\om'}] \subset \mathring{\sR}_d^{\omega'}$. In addition, we
can choose the paths so that $[w_{\om_1}, e_{\om'}] \cap 
[w_{\om_2}, e_{\om'}] = e_{\om'}$ for any pair $\om_1, \om_2 \prec \om'$
and $\om_1 \neq \om_2$. Moreover, for each $\om \in \paOm{\langle d]}{=1}$, 
we may arrange for the two paths, $[w_\om, e_{\om'_1}]$ and 
$[w_\om, e_{\om'_2}]$, to share the tangent vector at their common end $w_\om$, 
so that the path $[e_{\om'_1}, e_{\om'_2}]$ is transversal\footnote{The transversality is needed to insure the stability under perturbations of the map $\mathcal E$, defined below, with respect to the hypersurfaces $\{\mathring{\sR}_d^{\omega}\}$.} to the 
hypersurface $\mathring{\sR}_d^{\omega}$ at $w_\om$. This construction 
produces an embedding $\mathcal E: \Gr_d \to 
\cP_d^{\mathbf {c}\paOm{\langle d ]}{\geq 2}}$. In what follows, we do 
not distinguish between $\Gr_d$ and its image 
$\mathcal E(\Gr_d) \subset \cP_d^{\mathbf {c}\paOm{\langle d ]}{\geq 2}}$.

\begin{lemma}
  \label{lem:graph}
  The graph $\Gr_d$ 
  is 
  homotopy equivalent to a wedge of
  $\frac{d(d-2)}{4}$ circles if $d$ is even, and of
  $\frac{(d-1)^2}{4}$ circles if $d$ is odd. 
\end{lemma}
\begin{proof}
  A simple calculation shows that 
  $|\paOm{\langle d]}{=0}| = \lfloor \frac{d}{2} \rfloor+1$, and 
  $$|\paOm{\langle d]}{=1}| = \sum_{k=1}^{\lfloor \frac{d}{2} \rfloor}
  (2k-1) = \Big\lfloor\frac{d}{2} \Big\rfloor^2.$$
  For each $k \in  \{2,\ldots, d\}$ with the same parity as $d$, the vertex
  $(\underbrace{1,\ldots,1}_{k}) \in \paOm{\langle d]}{=0}$ is contained in
   the $k-1$ edges that lead to the vertices 
  $(\underbrace{1,\ldots,1}_{s}, 2,\underbrace{1,\ldots,1}_{k-2-s}) 
  \in \paOm{\langle d]}{=1}$, where $0 \leq s \leq k-2 \leq d-2$,
	and in the $k+1$ edges that lead to the vertices 
	$(\underbrace{1,\ldots,1}_{s},2,\underbrace{1,\ldots,1}_{k-s}) \in \paOm{\langle d]}{=1}$, where $0 \leq s \leq k \leq d-1$.
  For $d$ even, each case yields $\frac{d^2}{4}$ edges. 
  It is easily seen that the graph $\Gr_d$  is connected and hence it is 
	homotopy 
  equivalent to a wedge of circles. Now a simple 
  calculation of the Euler characteristic $\chi(\Gr_d)$ yields 
  $1-(\frac{d}{2} +1 + \frac{d^2}{4})+ 2\frac{d^2}{4} = \frac{d(d-2)}{4}$
  circles. The calculation for odd $d$ is analogous.
\end{proof}

Our next result is inspired by Arnold's \ref{th:Arnold1}. 
In \ref{fig1a}, for $d =6$, we illustrate it by exhibiting the 
cell structure in $\cP_6^{\bfc \paOm{\langle 6]}{\geq 2}}$ 
and its graph $\Gr_6$. 

\begin{theorem}\label{lm2b} 
  The space $\cP_d^{\bfc \paOm{\langle d]}{\geq 2}}$ is 
  homotopy equivalent to a 
  wedge of $\frac{d(d-2)}{4}$ circles for $d$ even, and to a 
  wedge of $\frac{(d-1)^2}{4}$ circles for $d$ odd.
      
  In particular, the fundamental group
  $\pi_1(\cP_d^{\bfc \paOm{\langle d]}{\geq 2}})$ is the 
  free group on $\frac{d(d-2)}{4}$ generators for even $d$ and 
  on $\frac{(d-1)^2}{4}$ generators for odd $d$, 
  and $\cP_d^{\bfc \paOm{\langle d]}{\geq 2}}$ is the corresponding 
  $K(\pi, 1)$-space.
\end{theorem} 

\begin{proof}  
  As an open subset of $\R^d$, the space 
  $\cP_d^{\mathbf{c}\paOm{\langle d]}{\geq 2}}$
  is paracompact. Now consider a finite open cover 
  $\mathcal X := \{X_\omega\}_{\omega \in \paOm{\langle d]}{=1}}$ of the space 
  $\cP_d^{\mathbf{c}\paOm{\langle d]}{\geq 2}}$, where 
  $X_\omega$ is the union of the $(d-1)$-cell 
  $\mathring\sR^\omega_d$ with the two adjacent $d$-cells that contain 
  $\mathring{\sR}^\omega_d$ in their boundary. 

  Each $X_\omega$ is open in $\cP_d^{\bfc \paOm{\langle d]}{\geq 2}}$. 
  Indeed, any point $x \in X_\omega$ either  (1) lies  in one of the two $d$-cells
  and thus has an open neighborhood in 
  $\cP_d^{\bfc \paOm{\langle d]}{\geq 2}}$ (contained in that 
  $d$-cell), or (2) $x \in \mathring{\sR}_d^\omega$, in which case it has on 
  open neighborhood $X_\omega$ in 
  $\cP_d^{\bfc \paOm{\langle d]}{\geq 2}}$.  

  By \cite[Lemma 2.4]{Ka} the attaching maps $\phi: D^d \to \cP_d$ of the 
  $d$-cells $D^d$ are injective on the $\phi$-preimage of each open 
  $(d-1)$ cell in $\cP_d$.
  This implies that $X_\omega$ retracts to the $\mathring{\sR}_d^\omega$,
  which, in turn, is contractible.
  For $i \geq 2$ and for pairwise distinct compositions 
  $\omega_1,\ldots, \omega_i \in \paOm{\langle d]}{=1}$, 
  the intersection $\bigcap_{j=1}^i X_{\omega_i}$ is either 
  empty, or is one of the open $d$-cells $\mathring{\sR}_d^\om$ for some
  $\om \in \paOm{\langle d]}{=0}$. It follows that, for $i \geq 1$ and for 
  compositions $\omega_1,\ldots, \omega_i \in \paOm{\langle d]}{=1}$, 
  the intersection $\bigcap_{j=1}^i X_{\omega_i}$ is either empty or 
  contractible.

  The preceding arguments show that the assumptions of 
  \cite[Corollary 4G.3]{Ha} are satisfied for the open
  covering $\mathcal X$ of $\cP_d^{\bfc \paOm{\langle d]}{\geq 2}}$.
  Hence $\cP_d^{\bfc \paOm{\langle d]}{\geq 2}}$ is
  homotopy equivalent to the nerve $N_\mathcal X$ of the covering $\mathcal X$. 
  We can identify  $N_\mathcal X$ with the simplicial complex, whose 
	simplices are the non-empty subsets $A$ of $\paOm{\langle d]}{=1}$
  such that $\bigcap_{\om \in A} X_\omega$ is non-empty. 
So the maximal simplices of the nerve $N_\mathcal X$ are in bijection with the elements of 
	$\paOm{\langle d]}{=0}$. 
	
	The maximal simplex, corresponding to 
	$\om' \in \paOm{\langle d]}{=0}$, contains all $\om \in \paOm{\langle d]}{= 1}$
  for which $\mathring{\sR}_d^{\om'} \subseteq X_\omega$.
  The intersection of the two maximal simplices, corresponding to
	$\omega',\omega'' \in \paOm{\langle d]}{= 0}$, is labeled by  all
	$\om \in \paOm{\langle d]}{=1}$ for which there are edges from 
  $\omega'$ to $\om$ and from $\omega''$ to $\om$ in $\Gr_d$.     

  The graph $\Gr_d$ can be covered by $1$-dimensional
  subcomplexes $Y_\omega$ , $\omega \in \paOm{\langle d]}{=1}$,
  where $Y_\omega$ 
  is the union of the two edges in $\Gr_d$, containing $\omega$. 
  It is easy to check that the nerve of this covering 
  $\mathcal Y := \{Y_\om\}_{\omega \in \paOm{\langle d]}{=1}}$ is
  again $N_\mathcal X$. 
  In fact, under the embedding 
  $\mathcal E: \Gr_d \to \cP_d^{\mathbf {c}\paOm{\langle d ]}{\geq 2}}$, 
  one has $Y_\om = X_\om \cap \mathcal E(\Gr_d)$.

  By \cite[Theorem 10.6]{Bj}, the nerve $N_\mathcal Y$ and the graph 
  $\Gr_d$ are homotopy equivalent.

  Moreover, by the proof of \cite[Corollary 4G.3]{Ha}, the following claim is 
  valid. Consider an embedding $Y \hookrightarrow X$ of a paracompact space 
  $Y$ into a paracompact space $X$ and a locally finite open covering 
  $\mathcal X = \{X_\a\}_\a$ of $X$. Put  
  $\mathcal Y = \{Y_\a := X_\a \cap Y\}_\a$. If, for any nonempty intersection 
  $\cap_i X_{\a_i}$, the intersection $\cap_i Y_{\a_i} \neq \emptyset$, 
  and both intersections are contractible, then the nerves $N_\mathcal X$ and 
  $N_\mathcal Y$ are naturally isomorphic as simplicial complexes, and 
  $Y \hookrightarrow X$ is a homotopy equivalence.
  
  We conclude that the embedding 
  $\mathcal E: \Gr_d \to \cP_d^{\bfc\paOm{\langle d ]}{\geq 2}}$ is a 
  homotopy equivalence. The result now follows from \ref{lem:graph}. 
\end{proof}

\begin{figure}
  \begin{center}
	  \includegraphics[width=0.6\textwidth]{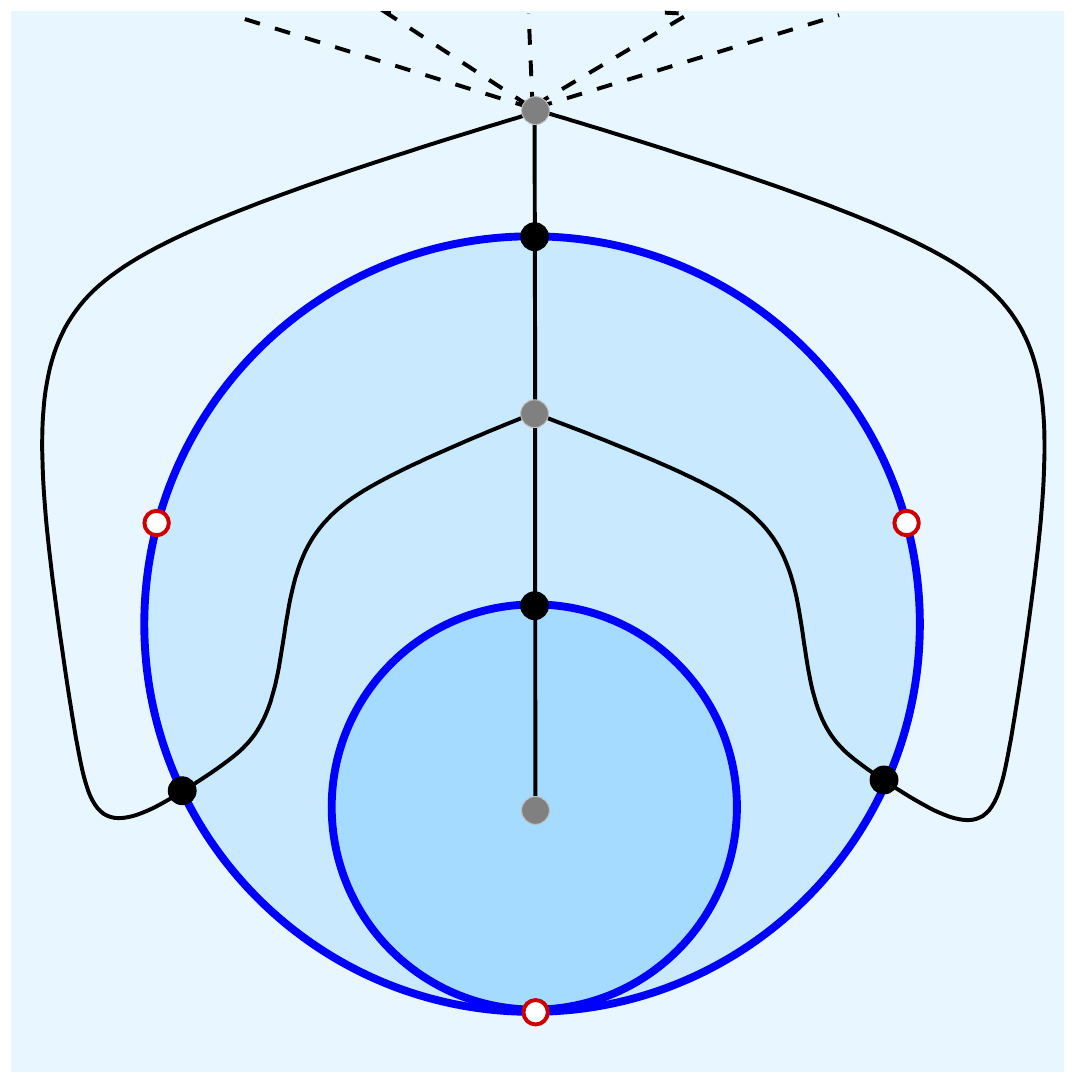}%
  \end{center}
  \setlength{\unitlength}{4144sp}%
  \begingroup\makeatletter\ifx\SetFigFont\undefined%
    \gdef\SetFigFont#1#2#3#4#5{%
    \reset@font\fontsize{#1}{#2pt}%
    \fontfamily{#3}\fontseries{#4}\fontshape{#5}%
    \selectfont}%
  \fi\endgroup%
  \begin{picture}(8193,8373)(1777,-7795)
	  \put(5500,1800){\makebox(0,0)[lb]{\smash{{\SetFigFont{12}{12}{\rmdefault}{\mddefault}{\updefault}{\color[rgb]{0.4,0.4,0.4}$\mathbf{()}$}%
    }}}}
	  \put(5500,3400){\makebox(0,0)[lb]{\smash{{\SetFigFont{12}{12}{\rmdefault}{\mddefault}{\updefault}{\color[rgb]{0.4,0.4,0.4}$\mathbf{(11)}$}%
    }}}}
	  \put(5750,4550){\makebox(0,0)[lb]{\smash{{\SetFigFont{12}{12}{\rmdefault}{\mddefault}{\updefault}{\color[rgb]{0.4,0.4,0.4}$\mathbf{(1111)}$}%
    }}}}
	  \put(5245,0732){\makebox(0,0)[lb]{\smash{{\SetFigFont{10}{10}{\rmdefault}{\mddefault}{\updefault}{\color[rgb]{0,0,0}$\mathbf{(22)}$}%
    }}}}
	  \put(5500,2700){\makebox(0,0)[lb]{\smash{{\SetFigFont{11}{11}{\rmdefault}{\mddefault}{\updefault}{\color[rgb]{0,0,0}$\mathbf{(2)}$}%
    }}}}
	  \put(6750,1500){\makebox(0,0)[lb]{\smash{{\SetFigFont{11}{11}{\rmdefault}{\mddefault}{\updefault}{\color[rgb]{0,0,0}$\mathbf{(211)}$}%
    }}}}
	  \put(3600,1500){\makebox(0,0)[lb]{\smash{{\SetFigFont{11}{11}{\rmdefault}{\mddefault}{\updefault}{\color[rgb]{0,0,0}$\mathbf{(112)}$}%
    }}}}
	  \put(3900,2700){\makebox(0,0)[lb]{\smash{{\SetFigFont{10}{10}{\rmdefault}{\mddefault}{\updefault}{\color[rgb]{0,0,0}$\mathbf{(13)}$}%
    }}}}
	  \put(6570,2700){\makebox(0,0)[lb]{\smash{{\SetFigFont{10}{10}{\rmdefault}{\mddefault}{\updefault}{\color[rgb]{0,0,0}$\mathbf{(31)}$}%
    }}}}
	  \put(5500,4150){\makebox(0,0)[lb]{\smash{{\SetFigFont{11}{11}{\rmdefault}{\mddefault}{\updefault}{\color[rgb]{0,0,0}$\mathbf{(121)}$}%
    }}}}
	  \put(5150,5000){\makebox(0,0)[lb]{\smash{{\SetFigFont{7}{7}{\rmdefault}{\mddefault}{\updefault}{\color[rgb]{0,0,0}$\mathbf{(11211)}$}%
    }}}}
	  \put(5800,5000){\makebox(0,0)[lb]{\smash{{\SetFigFont{7}{7}{\rmdefault}{\mddefault}{\updefault}{\color[rgb]{0,0,0}$\mathbf{(11121)}$}%
    }}}}
	  \put(6400,5000){\makebox(0,0)[lb]{\smash{{\SetFigFont{7}{7}{\rmdefault}{\mddefault}{\updefault}{\color[rgb]{0,0,0}$\mathbf{(11112)}$}%
    }}}}
	  \put(4500,5000){\makebox(0,0)[lb]{\smash{{\SetFigFont{7}{7}{\rmdefault}{\mddefault}{\updefault}{\color[rgb]{0,0,0}$\mathbf{(12111)}$}%
    }}}}
	  \put(3900,5000){\makebox(0,0)[lb]{\smash{{\SetFigFont{7}{7}{\rmdefault}{\mddefault}{\updefault}{\color[rgb]{0,0,0}$\mathbf{(21111)}$}%
    }}}}
  \end{picture}%
  ~\vskip-19cm
  \caption{A slice through the celluation of $\cP_6^{\bfc \paOm{\langle 6]}{\geq 2}}$ together with the graph $\Gr_6$, dual to the celluation (shown by black curves). Patterns $(22), (13), (31)$ label the strata of codimension $2$, while the arcs, labeled by $(2), (211), (121), (112)$, represent the strata of codimension $1$.}
    \label{fig1a}
\end{figure}

\begin{center}
\begin{figure}
  \begin{picture}(0,0)%
	  \includegraphics[width=\textwidth]{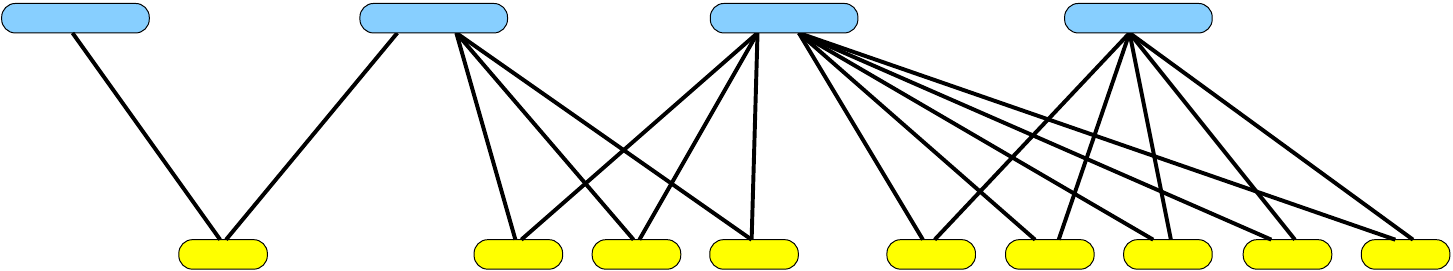}%
  \end{picture}%
   \setlength{\unitlength}{4144sp}%
  \begingroup\makeatletter\ifx\SetFigFont\undefined%
  \gdef\SetFigFont#1#2#3#4#5{%
  \reset@font\fontsize{#1}{#2pt}%
  \fontfamily{#3}\fontseries{#4}\fontshape{#5}%
  \selectfont}%
  \fi\endgroup%
  \begin{picture}(10599,2109)(664,-4161)
	  \put(1000,-2920){\rotatebox{0.0}{\makebox(0,0)[lb]{\smash{{\SetFigFont{7}{7}{\rmdefault}{\mddefault}{\updefault}{\color[rgb]{0,0,0}$\mathbf{()}$}%
  }}}}}
	  \put(2730,-2920){\rotatebox{0.0}{\makebox(0,0)[lb]{\smash{{\SetFigFont{7}{7}{\rmdefault}{\mddefault}{\updefault}{\color[rgb]{0,0,0}$\mathbf{(11)}$}%
  }}}}}
	  \put(4400,-2920){\rotatebox{0.0}{\makebox(0,0)[lb]{\smash{{\SetFigFont{7}{7}{\rmdefault}{\mddefault}{\updefault}{\color[rgb]{0,0,0}$\mathbf{(1111)}$}%
  }}}}}
	  \put(6130,-2920){\rotatebox{0.0}{\makebox(0,0)[lb]{\smash{{\SetFigFont{7}{7}{\rmdefault}{\mddefault}{\updefault}{\color[rgb]{0,0,0}$\mathbf{(\!111111\!)}$}%
  }}}}}

	  \put(1700,-4115){\rotatebox{0.0}{\makebox(0,0)[lb]{\smash{{\SetFigFont{7}{7}{\rmdefault}{\mddefault}{\updefault}{\color[rgb]{0,0,0}$\mathbf{(2)}$}%
  }}}}}
	  \put(3110,-4115){\rotatebox{0.0}{\makebox(0,0)[lb]{\smash{{\SetFigFont{7}{7}{\rmdefault}{\mddefault}{\updefault}{\color[rgb]{0,0,0}$\mathbf{(121)}$}%
  }}}}}
	  \put(3700,-4115){\rotatebox{0.0}{\makebox(0,0)[lb]{\smash{{\SetFigFont{7}{7}{\rmdefault}{\mddefault}{\updefault}{\color[rgb]{0,0,0}$\mathbf{(211)}$}%
  }}}}}
	  \put(4290,-4115){\rotatebox{0.0}{\makebox(0,0)[lb]{\smash{{\SetFigFont{7}{7}{\rmdefault}{\mddefault}{\updefault}{\color[rgb]{0,0,0}$\mathbf{(112)}$}%
  }}}}}
	  \put(5130,-4115){\rotatebox{0.0}{\makebox(0,0)[lb]{\smash{{\SetFigFont{7}{7}{\rmdefault}{\mddefault}{\updefault}{\color[rgb]{0,0,0}$\mathbf{(\!11211\!)}$}%
  }}}}}
	  \put(5730,-4115){\rotatebox{0.0}{\makebox(0,0)[lb]{\smash{{\SetFigFont{7}{7}{\rmdefault}{\mddefault}{\updefault}{\color[rgb]{0,0,0}$\mathbf{(\!12111\!)}$}%
  }}}}}
	  \put(6320,-4115){\rotatebox{0.0}{\makebox(0,0)[lb]{\smash{{\SetFigFont{7}{7}{\rmdefault}{\mddefault}{\updefault}{\color[rgb]{0,0,0}$\mathbf{(\!11121\!)}$}%
  }}}}}
	  \put(6925,-4115){\rotatebox{0.0}{\makebox(0,0)[lb]{\smash{{\SetFigFont{7}{7}{\rmdefault}{\mddefault}{\updefault}{\color[rgb]{0,0,0}$\mathbf{(\!21111\!)}$}%
  }}}}}
	  \put(7520,-4115){\rotatebox{0.0}{\makebox(0,0)[lb]{\smash{{\SetFigFont{7}{7}{\rmdefault}{\mddefault}{\updefault}{\color[rgb]{0,0,0}$\mathbf{(\!11112\!)}$}%
  }}}}}
  \end{picture}%
  \caption{The graph $\Gr_6$ as part of the poset $\bfOm_{\langle 6]}$}
	\label{fig2a}
\end{figure}
\end{center}

\begin{center}
\begin{figure}
  \begin{picture}(0,0)%
	  \includegraphics[width=\textwidth]{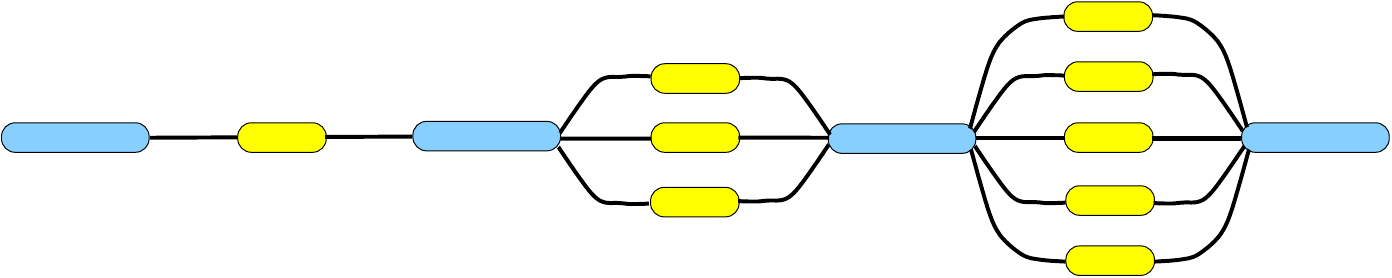}%
  \end{picture}%
   \setlength{\unitlength}{4144sp}%
  \begingroup\makeatletter\ifx\SetFigFont\undefined%
  \gdef\SetFigFont#1#2#3#4#5{%
  \reset@font\fontsize{#1}{#2pt}%
  \fontfamily{#3}\fontseries{#4}\fontshape{#5}%
  \selectfont}%
  \fi\endgroup%
  \begin{picture}(10599,2109)(664,-4161)
	  \put(1000,-3460){\rotatebox{0.0}{\makebox(0,0)[lb]{\smash{{\SetFigFont{7}{7}{\rmdefault}{\mddefault}{\updefault}{\color[rgb]{0,0,0}$\mathbf{()}$}%
  }}}}}
	  \put(2060,-3460){\rotatebox{0.0}{\makebox(0,0)[lb]{\smash{{\SetFigFont{7}{7}{\rmdefault}{\mddefault}{\updefault}{\color[rgb]{0,0,0}$\mathbf{(2)}$}%
  }}}}}
	  \put(3060,-3460){\rotatebox{0.0}{\makebox(0,0)[lb]{\smash{{\SetFigFont{7}{7}{\rmdefault}{\mddefault}{\updefault}{\color[rgb]{0,0,0}$\mathbf{(11)}$}%
  }}}}}
	  \put(4150,-3460){\rotatebox{0.0}{\makebox(0,0)[lb]{\smash{{\SetFigFont{7}{7}{\rmdefault}{\mddefault}{\updefault}{\color[rgb]{0,0,0}$\mathbf{(121)}$}%
  }}}}}
	  \put(4150,-3150){\rotatebox{0.0}{\makebox(0,0)[lb]{\smash{{\SetFigFont{7}{7}{\rmdefault}{\mddefault}{\updefault}{\color[rgb]{0,0,0}$\mathbf{(211)}$}%
  }}}}}
	  \put(4150,-3800){\rotatebox{0.0}{\makebox(0,0)[lb]{\smash{{\SetFigFont{7}{7}{\rmdefault}{\mddefault}{\updefault}{\color[rgb]{0,0,0}$\mathbf{(112)}$}%
  }}}}}
	  \put(5160,-3460){\rotatebox{0.0}{\makebox(0,0)[lb]{\smash{{\SetFigFont{7}{7}{\rmdefault}{\mddefault}{\updefault}{\color[rgb]{0,0,0}$\mathbf{(1111)}$}%
  }}}}}
	  \put(6270,-3460){\rotatebox{0.0}{\makebox(0,0)[lb]{\smash{{\SetFigFont{7}{7}{\rmdefault}{\mddefault}{\updefault}{\color[rgb]{0,0,0}$\mathbf{(\!11211\!)}$}%
  }}}}}
	  \put(6270,-3140){\rotatebox{0.0}{\makebox(0,0)[lb]{\smash{{\SetFigFont{7}{7}{\rmdefault}{\mddefault}{\updefault}{\color[rgb]{0,0,0}$\mathbf{(\!12111\!)}$}%
  }}}}}
	  \put(6270,-3790){\rotatebox{0.0}{\makebox(0,0)[lb]{\smash{{\SetFigFont{7}{7}{\rmdefault}{\mddefault}{\updefault}{\color[rgb]{0,0,0}$\mathbf{(\!11121\!)}$}%
  }}}}}
	  \put(6270,-2830){\rotatebox{0.0}{\makebox(0,0)[lb]{\smash{{\SetFigFont{7}{7}{\rmdefault}{\mddefault}{\updefault}{\color[rgb]{0,0,0}$\mathbf{(\!21111\!)}$}%
  }}}}}
	  \put(6270,-4110){\rotatebox{0.0}{\makebox(0,0)[lb]{\smash{{\SetFigFont{7}{7}{\rmdefault}{\mddefault}{\updefault}{\color[rgb]{0,0,0}$\mathbf{(\!11112\!)}$}%
  }}}}}
	  \put(7310,-3460){\rotatebox{0.0}{\makebox(0,0)[lb]{\smash{{\SetFigFont{7}{7}{\rmdefault}{\mddefault}{\updefault}{\color[rgb]{0,0,0}$\mathbf{(\!111111\!)}$}%
  }}}}}
  \end{picture}%
  \caption{The graph $\Gr_6$, drawn as it is embedded in $\cP^{\bfc\paOm{\langle 6]}{\leq 2}}_6$}
	\label{fig2}
\end{figure}
\end{center}

Instead of $\Gr_d$ we can also consider the graph $\Gr_d'$ in which any 
pair of edges $e,e'$, forming a path that joins two different nodes  labeled 
by elements from $\paOm{\langle d]}{=0}$, is replaced by a single edge labeled by the unique $\omega \in \paOm{\langle d]}{=1}$ in the intersection of $e$ and $e'$.
Thus $\Gr_d$ is the graph subdivision of $\Gr_d'$ and in particular, we have
$\pi_1(\Gr_d) \cong \pi_1(\Gr_d')$ when both are considered as $1$-dimensional
simplicial complexes.
We will use the graph $\Gr_d'$ in a second slightly different and perhaps more 
natural approach to the calculation of 
$\pi_1(\cP_d^{\bfc\paOm{\langle d]}{\geq 2}})$.\smallskip
 
For that we introduce the alphabet $\mathsf A$ with 
letters
$\omega_{ij} = (\underbrace{1,\dots,1}_{i},2,\underbrace{1,\ldots, 1}_j)$, 
where $0 \leq i,j$, $i+j \leq d-2$ and $i+j \equiv d \mod 2$. 

\smallskip
(See illustrations in \ref{fig1a}, \ref{fig2a} and \ref{fig2}).  

\medskip
We consider now the free group, 
generated by the letters from $\mathsf A$. The letters 
are in obvious bijection with the edges of the graph $\Gr_d'$. For a letter 
$\omega \in \mathsf A$,  we use 
$\omega^+= \omega$ or its inverse $\omega^-= \omega^{-1}$ to capture 
the orientations of these edges 
and to represent elements of the free group
generated by $\mathsf A$ by words. 
We write $\mathsf A^{\pm}$ for the set of letters $\omega^\pm$
where $\omega \in \mathsf A$. 
We use the convention that 
the sign ''+'' corresponds to the orientation from the vertex with less 
ones towards the vertex with more ones, and the sign ''-'' corresponds to 
the opposite orientation. 

We choose $()$ as the basepoint of $\Gr_d'$ if $d$ is even, and 
$(1)$ as the basepoint if $d$ is odd.
In particular, we can consider $\pi_1(\Gr_d')$ as a subgroup of the free
group over $\mathsf A$. 

Next, we introduce a class of words which will be used to define canonical 
representatives of elements from $\pi_1(\Gr_d')$.

\begin{definition}\label{admissible} 
  For $d\geq 1$, we say that a word $w$ in the alphabet $\mathsf A^\pm$ is 
  {\it admissible} if it is either empty or satisfies the following two conditions:

  \begin{itemize}
    \item[(a$_e$)] if $d$ is even, then 
       $w$ starts with the letter $(2)^+$ and ends with the letter 
       $(2)^-$;

    \item[(a$_o$)] if $d$ is odd, then $w$ starts with either the letter $(12)^+$ or with the letter $(21)^+$ and ends either with  $(12)^-$ or with $(21)^-$. 

    \item[(b)] any pair $(\om_1^\pm, \om_2^\pm)$ of two consecutive letters 
       has the property that the number of $1$'s in $\om_1$ and $\om_2$ 
       either coincide or differs by two. In the former case, the signs of 
       the letters are different, and in the latter case, the signs 
       should be as follows. If $\om_1$ has less ones than $\om_2$, then we 
       only allow the pair $(\om_1^+, \om_2^+)$;  if $\om_1$ has 
       more ones than $\om_2$, then we only allow the pair 
       $(\om_1^-, \om_2^-)$.
  \end{itemize}
\end{definition}

Clearly, admissible words are in bijection with based loops in the graph  
$\Gr_d'$.
Here are two examples of admissible words representing based loops, the first one for even $d \geq 6$ and 
second one for odd $d \geq 7$: 
$$w_1=(2)^+\,(112)^+ \, (121)^- \, (211)^+ \, (11112)^+ \,(12111)^- \,(121)^- \, (2)^-,$$ 
$$w_2= (21)^+ \, (1112)^+ \, (112111)^+ \, (111211)^- \, (1211)^- \, (12)^-.$$

We observe that a word is admissible if and only if its reduction (in the sense
of combinatorial group theory) is admissible. \smallskip

Let $\mathcal G_d$ be the set of reduced admissible words over $\mathsf A^{\pm}$. 
To each $\omega_{ij}^{+} \in \mathsf A^\pm$, we associate the word 
$\gamma_{ij}$ given by
\begin{itemize}
	\item $$\gamma_{ij} = \omega_{0,0}^+\,\omega_{0,2}^+\,\cdots \, \omega_{0,i+j-2}^+\, \omega_{ij}^+\, 
                          \omega_{0,i+j}^-\, \cdots \,\omega_{0,2}^-\,  \omega_{0,0}^-$$ 
if $d$ is even and
		  \item $$\gamma_{ij} = \omega_{0,1}^+\,\omega_{0,3}^+\, \cdots \,\omega_{0,i+j-2}^+ \,\omega_{ij}^+\,  
                          \omega_{0,i+j}^-\, \cdots \, \omega_{0,3}^-\,  \omega_{0,1}^-$$ 

if $d$ is odd.
\end{itemize}

All $\gamma_{ij}$ are admissible, and if $i \neq 0$ then $\gamma_{ij}$ is reduced. We notice that each $\gamma_{ij}$ represents a single \emph{loop} in $\Gr_d'$.

The following is then a standard fact about graphs considered as $1$-dimensional CW-complexes.
 
\begin{lemma}\label{lm:group} 
  The set $\mathcal G_d$ is a subgroup of the free group over the alphabet 
  $\mathsf A$, isomorphic to $\pi_1(\Gr_d')$. 
  In particular, $\mathcal G_d$ is a free group itself.
  The $\gamma_{ij}$ for $0 \leq i,j$, $i \neq 0$ and $i+j \leq d-2$ form
  a minimal generating set of $\mathcal G_d$.  
\end{lemma} 

Note each $\gamma_{0,j}$ is $1$ in $\mathcal G_d$. But for our geometric 
picture and for the sake of a less technical presentation, we will continue 
working with the $\gamma_{0,j}$.

Now we would like to relate this representation of $\pi_1(\Gr_d')$ to 
$\pi_1(\cP_d^{\bfc\paOm{\langle d]}{\geq 2}})$. 
We use the notations from the proof of \ref{lm2b}. 

For $\omega_{ij} \in \mathsf A$,  
consider the codimension one 
``wall'' $\mathring\sR^{\omega_{ij}}_d$.
The walls divide $\cP_d^{\bfc\paOm{\langle d]}{\geq 2}}$ into a 
set of open 
$d$-cells $\{\mathring\sR^\omega_d\}$ for
$\om = (\underbrace{1,\ldots,1}_{i}) \in \paOm{\langle d]}{=0}$, 
where $0 \leq i \leq d$ is even when $d$ is even and odd when $d$ is odd.

We orient each wall $\mathring\sR^{\omega_{ij}}_d$ in such a way that crossing 
it in the preferred direction increases the number of simple real 
roots by $2$. We consider this as a ``$+$''-crossing and the  crossing in the
opposite direction as a ``$-$''-crossing. 
This notation goes along with our convention for orienting the edges in 
$\Gr_d'$. 

Consider an oriented loop 
$\gamma: S^1 \to \cP_d^{\bfc\paOm{\langle d]}{\geq 2}}$,
based in $\mathring\sR^{()}_d$ if $d$ is even and in 
$\mathring\sR^{(1)}_d$ if $d$ is odd.
By the general position arguments, we may assume that $\gamma$ is smooth and 
transversal to each wall. In particular, since $S^1$ is 
compact, the (transversal) intersection of $\gamma(S^1)$ with each wall 
is a finite set. As we move along $\gamma$, we record each 
transversal crossing 
$\gamma \cap \mathring{\sR}_d^{\omega_{ij}}$ and the direction of the 
crossing by the corresponding letter in $\mathsf A^{\pm}$. 

\medskip

\begin{lemma}\label{lm:words} 
  The homotopy classes of based loops $\gamma \subset 
  \cP_d^{\bfc\paOm{\langle d]}{\geq 2}}$ and the homotopy classes of 
  based loops in $\Gr_d'$ are in 
  one-to-one correspondence with the reduced admissible words 
  over the alphabet $\mathsf A^\pm$. 
  In addition, concatenation of loops corresponds to the concatenation 
  and reduction of words.
  
  In particular, 
  $\pi_1(\cP_d^{\bfc\paOm{\langle d]}{\geq 2}}) \simeq \mathcal G_d
  \simeq \pi_1(\Gr_d')$ is a free group.
\end{lemma}

\begin{proof}[Sketch of Proof:]
  For $\pi_1(\Gr_d')$ the assertion already follows from \ref{lm:group} and
  the arguments preceding the lemma. 

  Our choice of basepoint in $\cP_d^{\bfc\paOm{\langle d]}{\geq 2}}$ 
  corresponds  exactly to the choice of the base vertices $()$ for $d$ even 
  and $(1)$ for $d$ odd for admissible words in $\Gr_d'$. The above arguments 
  and the conditions for 
  admissibility of words over the alphabet $\mathsf A^\pm$ show that
  based loops in 
  $\cP_d^{\bfc\paOm{\langle d]}{\geq 2}}$ 
  are represented by admissible words. Reduction is easily seen 
  to be a homotopy equivalence. Moreover,  loops corresponding to different 
  reduced admissible words are not homotopy equivalent. 
  Since the concatenation of admissible words is admissible, the remaining
  assertion now follows.
\end{proof}

\begin{remark}
  Let $X$ be a compact path-connected Hausdorff space and $C(X)$ the space of continuous functions on $X$. 
  Let $P(u, x) = u^d + a_{d-1} u^{d-1} + \ldots + a_1 u + a_0$ be a polynomial in $u$ with coefficients $a_i \in C(X)$. 
  Assume, that for each $x \in X$, $P(u, x) \in \mathcal P_d^{\bfc\paOm{\langle d]}{\geq 2}}$; that is 
  $P(u, x)$ has only simple real roots and no more than a single root of multiplicity $2$. 
  
 Thus, every such $P(u, x)$ generates a continuous map $\a_P: X \to \mathcal P_d^{\bfc\paOm{\langle d]}{\geq 2}}$. Since 
  a subgroup of a free group is free or trivial, it follows that when $\pi_1(X)$ does not admit an epimorphism with the 
  free target group, then the induced map $$(\a_P)_\ast: \pi_1(X) \to \pi_1(\mathcal P_d^{\bfc\paOm{\langle d]}{\geq 2}})$$
  must be trivial. For example, this is the case for any $X$ with a finite fundamental group. 
  
  Here is one implication of the triviality of $(\a_P)_\ast$: for any loop $\delta: S^1 \to X$, the loop $(\a_P)_\ast(\delta)$ hits 
  the non-singular part of the discriminant variety $\mathcal D_d \subset \mathcal P_d$ so that the admissible word in the alphabet $\mathsf A^\pm$, 
  generated by $(\a_P)_\ast(\delta)$, can be reduced to the identity $1$. 
  \smallskip 
\end{remark}

\subsection{The fundamental group $\pi_1(\cP_d^{\bfc\Theta})$ for any closed poset $\Theta$ with the locus $\cP_d^{\Theta} \subset \cP_d$ of codimension $ \geq 2$}
 
This section describes how the fundamental group changes if we add to the space 
$\cP_d^{\bfc\paOm{\langle d]}{\geq 2}}$ some strata 
of codimension $2$. 
Each of these strata will create a relation in the fundamental group of 
the new space under construction. We provide an explicit description of these  
relations and show that, in many cases, the resulting quotient is still a \emph{free} group.
\smallskip 

We start with the following simple statement. 

\begin{lemma}\label{lm2bN} 
  Let $\Theta \subset \bfOm_{\langle d]}$ be a closed poset such that 
  $\Theta \subset \paOm{\langle d]}{\geq k}$. Then 
  the homotopy groups $\pi_i(\cP_d^{\bfc\Theta})$ vanish for all $i < k-1$. 
  In particular,  $\pi_1(\cP_d^{\bfc\Theta})$ vanishes,  provided 
  $\Theta \subset \paOm{\langle d]}{\geq 3}$. As a special case, we have 
  $\pi_1(\cP_d^{\bfc\paOm{\langle d]}{\geq 3}}) = 0$.
\end{lemma} 

\begin{proof}   
  We observe that if $\Theta \subset \paOm{\langle d]}{\geq k}$, then 
  $\textup{codim}(\bar{\cP}_d^{\Theta}, \bar{\cP}_d) \geq k$. Therefore, 
  by the general position argument, 
  $\pi_i(\cP_d^{\bfc\Theta}) = 0$ for all $i < k-1$. In particular, 
  $\pi_1(\cP_d^{\bfc\Theta}) = 0$,  provided that 
  $\Theta \subset \paOm{\langle d]}{\geq 3}$.
\end{proof} 
  
\begin{figure}

  \begin{picture}(0,0)%
    \hskip1.0cm\includegraphics[width=0.9\textwidth]{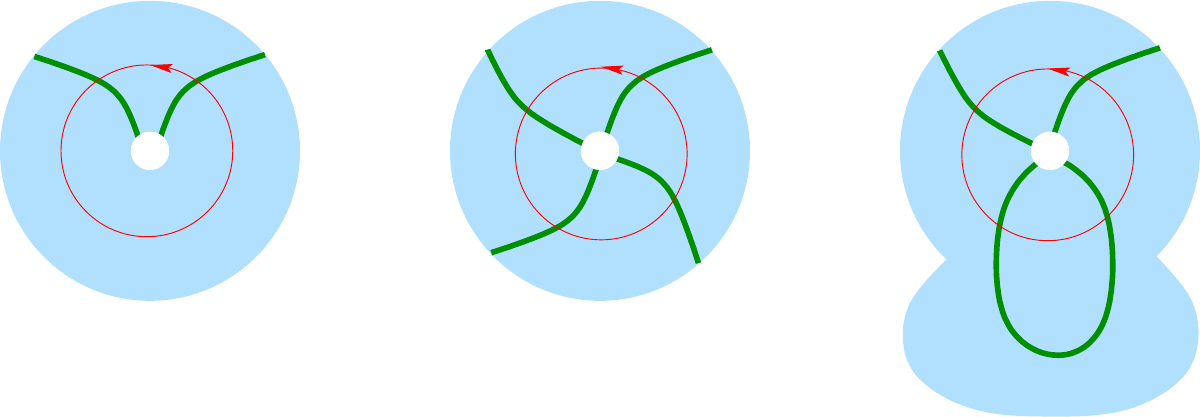}%
  \end{picture}%
  \setlength{\unitlength}{3947sp}%
  \begingroup\makeatletter\ifx\SetFigFont\undefined%
    \gdef\SetFigFont#1#2#3#4#5{%
    \reset@font\fontsize{#1}{#2pt}%
  \fontfamily{#3}\fontseries{#4}\fontshape{#5}%
  \selectfont}%
  \fi\endgroup%
  \begin{picture}(9600,3326)(1801,-3688)
    \put(8100,-3550){\makebox(0,0)[lb]{\smash{{\SetFigFont{6}{6}{\rmdefault}{\mddefault}{\updefault}{\color{blue}${\bf (1\,1\,1\,1)}$}%
}}}}
    \put(8200,-2900){\makebox(0,0)[lb]{\smash{{\SetFigFont{6}{6}{\rmdefault}{\mddefault}{\updefault}{\color{blue}${\bf (1\,1)}$}%
}}}}
    \put(8000,-1500){\makebox(0,0)[lb]{\smash{{\SetFigFont{6}{6}{\rmdefault}{\mddefault}{\updefault}{\color{blue}${\bf (1\,1\,1\,1\,1\,1)}$}%
}}}}
    \put(8650,-3200){\makebox(0,0)[lb]{\smash{{\SetFigFont{6}{6}{\rmdefault}{\mddefault}{\updefault}{\color{forestgreen}${\bf (1\,2\,1)}$}%
}}}}
    \put(8080,-2200){\makebox(0,0)[lb]{\smash{{\SetFigFont{6}{6}{\rmdefault}{\mddefault}{\updefault}{\color[rgb]{0,0,0}${\bf (1\,2\,2\,1)}$}%
}}}}
    \put(8900,-1550){\makebox(0,0)[lb]{\smash{{\SetFigFont{6}{6}{\rmdefault}{\mddefault}{\updefault}{\color{forestgreen}${\bf (1\,1\,1\,2\,1)}$}%
}}}}
    \put(7200,-1550){\makebox(0,0)[lb]{\smash{{\SetFigFont{6}{6}{\rmdefault}{\mddefault}{\updefault}{\color{forestgreen}${\bf (1\,1\,1\,2\,1)}$}%
}}}}
    \put(6200,-2900){\makebox(0,0)[lb]{\smash{{\SetFigFont{6}{6}{\rmdefault}{\mddefault}{\updefault}{\color{forestgreen}${\bf (2\,1\,1)}$}%
}}}}
    \put(4750,-2800){\makebox(0,0)[lb]{\smash{{\SetFigFont{6}{6}{\rmdefault}{\mddefault}{\updefault}{\color{forestgreen}${\bf (1\,2\,1)}$}%
}}}}
    \put(6200,-1550){\makebox(0,0)[lb]{\smash{{\SetFigFont{6}{6}{\rmdefault}{\mddefault}{\updefault}{\color{forestgreen}${\bf (2\,1\,1\,1\,1)}$}%
}}}}
    \put(4750,-1550){\makebox(0,0)[lb]{\smash{{\SetFigFont{6}{6}{\rmdefault}{\mddefault}{\updefault}{\color{forestgreen}${\bf (1\,1\,1\,2\,1)}$}%
}}}}
    \put(5500,-2200){\makebox(0,0)[lb]{\smash{{\SetFigFont{6}{6}{\rmdefault}{\mddefault}{\updefault}{\color[rgb]{0,0,0}${\bf (2\,1\,2\,1)}$}%
}}}}
    \put(5550,-1500){\makebox(0,0)[lb]{\smash{{\SetFigFont{6}{6}{\rmdefault}{\mddefault}{\updefault}{\color{blue}${\bf (1\,1)}$}%
}}}}
    \put(5400,-2900){\makebox(0,0)[lb]{\smash{{\SetFigFont{6}{6}{\rmdefault}{\mddefault}{\updefault}{\color{blue}${\bf (1\,1\,1\,1\,1\,1)}$}%
}}}}
    \put(5000,-2400){\makebox(0,0)[lb]{\smash{{\SetFigFont{6}{6}{\rmdefault}{\mddefault}{\updefault}{\color{blue}${\bf (1\,1\,1\,1)}$}%
}}}}
    \put(6020,-2100){\makebox(0,0)[lb]{\smash{{\SetFigFont{6}{6}{\rmdefault}{\mddefault}{\updefault}{\color{blue}${\bf (1\,1\,1\,1)}$}%
}}}}
    \put(2900,-2200){\makebox(0,0)[lb]{\smash{{\SetFigFont{6}{6}{\rmdefault}{\mddefault}{\updefault}{\color[rgb]{0,0,0}${\bf (1\,3\,1\,1)}$}%
}}}}
    \put(2200,-1550){\makebox(0,0)[lb]{\smash{{\SetFigFont{6}{6}{\rmdefault}{\mddefault}{\updefault}{\color{forestgreen}${\bf (1\,2\,1\,1\,1)}$}%
}}}}
    \put(3550,-1550){\makebox(0,0)[lb]{\smash{{\SetFigFont{6}{6}{\rmdefault}{\mddefault}{\updefault}{\color{forestgreen}${\bf (1\,1\,2\,1\,1)}$}%
}}}}
    \put(2900,-1500){\makebox(0,0)[lb]{\smash{{\SetFigFont{6}{6}{\rmdefault}{\mddefault}{\updefault}{\color{blue}${\bf (1\,1\,1\,1)}$}%
}}}}
    \put(2800,-2850){\makebox(0,0)[lb]{\smash{{\SetFigFont{6}{6}{\rmdefault}{\mddefault}{\updefault}{\color{blue}${\bf (1\,1\,1\,1\,1\,1)}$}%
}}}}
  \end{picture}%
  \caption{The normal disks to the strata $\mathcal P_6^{(1311)}$ (left), $\mathcal P_6^{(2121)}$ (middle), and $\mathcal P_6^{(1221)}$ (right) and the loops bounding them  (in red).} 
\label{fig3}
\end{figure}

Further, we observe that, by the Alexander duality and the Hurewicz Theorem, 
for any closed $\Theta\subset \bfOm_{\langle d]}$, 
a minimal generating set of $\pi_1(\cP_d^{\bfc\Theta})$ contains at least 
$\rank(\bar H_{d-2}(\bar{\cP}_d^{\Theta}; \Z))$ elements.

Given a closed poset $\Theta \subseteq \paOm{\langle d]}{\geq 2}$, we consider 
two disjoint sets: 
\begin{eqnarray}\label{Om=2}
\Theta_{=2} = \paOm{\langle d]}{=2} \cap \Theta,  \quad \bfc \Theta_{=2} = \paOm{\langle d]}{=2} \setminus \Theta.
\end{eqnarray}
By definition, $\Theta_{=2}$ and $\bfc \Theta_{=2}$  both consist of $\omega$'s which  have some parts $1$ and either a \emph{single} entry $3$, or \emph{two} $2$s.

Now consider a loop $\gamma$ in $\cP_d^{\bfc\Theta}$. It bounds a $2$-disk 
$D$ in $\cP_d$. By a general position argument, we may assume that 
$D$ avoids all strata $\sR_d^\omega$ with $|\omega|' \geq 3$ and that
if $\gamma$ intersects a stratum $\mathring{\sR}_d^\omega$ for some 
$|\omega|' = 2$, then this intersection is transversal. 

Consider a fixed $\omega$ with $|\omega|' = 2$ such that 
$\gamma$ 
intersects $\mathring{\sR}_d^\omega$. For 
$x \in \mathring{\sR}_d^\omega$, consider a small $2$-disk $x \in D_x \subseteq
\cP_d^{\bfc\paOm{\langle d]}{\geq 3}}$ 
transversal to $\mathring{\sR}_d^\omega$ which intersects only
cells $\mathring{\sR}_d^ {\omega'}$ for $\omega' \succeq \omega$. 
Let  $\kappa_x := \d D_x \subset \cP_d^{\bfc\paOm{\langle d]}{\geq 2}}$ be 
the loop bounding $D_x$ (see \ref{fig3}). 
Since $\mathring{\sR}_d^\omega$ is contractible, for each 
$x \in D \cap \mathring{\sR}_d^\omega$, the homotopy class of $\kappa_x$ 
depends  only on $\omega$ (up to inversion). 
Hence we can speak of $\kappa_\omega$. 
 
We choose a path $\beta \subset \cP_d^{\bfc\paOm{\langle d]}{\geq 2}}$ that connects our base point in $\mathring{\sR}_d^{()}$ or $\mathring{\sR}_d^{(1)}$ 
with $\kappa_\omega$. By a general position argument we can again assume that 
$\beta$ intersects all codimension $1$ strata transversally. 
Let $\kappa_\omega^\bullet := \beta^{-1}\circ \kappa_\omega \circ \beta$ be the loop that starts at the base point, follows a path $\beta$, then 
traverses $\kappa_\omega$ once, and returns to the base point following 
$\beta^{-1}$.
By recording the intersections of $\kappa_\omega^\bullet$ with the codimension $1$ strata, indexed by the $\omega_{ij}$, together with the direction of the intersections 
we obtain an admissible word in $\mathsf A^\pm$. 
By \ref{lm2bN}, the homotopy class of $\kappa_\omega^\bullet$ in $\cP_d^{\bfc\paOm{\langle d]}{\geq 2}}$ depends only on this word. 
In particular, we can express this word as a product of words $\gamma_{ij}$
or their inverses for the $\omega_{ij}$ intersected by $\kappa_\omega$. 

Evidently, if $\omega \in \bfc\Theta_{=2}$, the loop $\kappa_\omega$ and thus 
the loop $\kappa^\bullet_\omega$ are contractible in $\cP_d^{\bfc\Theta}$.  
The following lemma gives a precise analysis of the contractible 
$\kappa_\omega^\bullet$ and shows that the triviality of the corresponding 
words provides a presentation of  $\pi_1(\cP_d^{\bfc\Theta})$ 
generalizing \ref{lm2b}.

\begin{proposition}\label{lem2.14} 
  For any closed poset $\Theta \subset \paOm{\langle d]}{\geq 2}$,
  the fundamental group $\pi_1(\cP_d^{\bfc\Theta})$ is a quotient of the free group 
  $\mathcal G_d$ 
  by the normal subgroup, generated by the following relations, one relation for every $\om \in \bfc\Theta_{=2}$.

  \begin{itemize}
    \item[\bf(3)] For each $\om = (\underbrace{1\cdots 1}_i 3 \underbrace{1\cdots1}_j ) \in \bfc \Theta_{=2}$
       the corresponding relation is 
       $$\gamma_{i,j+1} \, \gamma_{i+1,j}^{-1} = 1$$ 

     \item[\bf(22)]  For each $\om = (\underbrace{1\cdots1}_i 2 \underbrace{1\cdots 1}_j 2 \underbrace{1 \cdots 1}_\ell ) \in \bfc\Theta_{=2}$ 
       the corresponding 
       relation is $$\gamma_{i+j,\ell}\, \gamma_{i+j+2,\ell}\, \gamma_{i,j+\ell+2}^{-1} \,\gamma_{i,j+\ell}^{-1} = 1.$$  
  \end{itemize}
\end{proposition}

Recall, that whenever a $\gamma_{0,j}$ appears in one of the relations that
it can be omitted by $\gamma_{0j} = 1$. 
Note also that, in case (22), when $j = 0$,  the relation 
$\gamma_{i+j,\ell}\, \gamma_{i+j+2,\ell}\, \gamma_{i,j+\ell+2}^{-1} \,
\gamma_{i,j+\ell}^{-1}$ is conjugate to the relation
$\gamma_{i+j+2,\ell}\,\gamma_{i,j+\ell+2}^{-1}$.  
 
\begin{proof}  
  First we show that the relations (3) and (22) lie in the kernel of the homomorphism
  \begin{align}
    \label{eq:inc} 
    \pi_1(\cP_d^{\bfc\paOm{\langle d]}{\leq 2}}) & \rightarrow \pi_1(\cP_d^{\bfc\Theta}) 
  \end{align}
  induced by the inclusion of spaces. 
  By the arguments preceding the lemma, it suffices to show that each relation
  corresponds to an admissible word representing $\kappa_\om^\bullet$ for some 
  $\om \in \paOm{\langle d]}{= 2}$. 

  \begin{itemize}
    \item[\bf(3)] 
        For each $\om = (\underbrace{1,\cdots , 1}_i, 3, \underbrace{1,\cdots, 1}_j ) \in \bfc \Theta_{=2}$,
        we notice that $\mathring{\sR}^\om_d$ lies only in the boundaries of
        the codimension one cells $\mathring{\sR}^{\omega_{i+1,j}}_d$ and 
        $\mathring{\sR}^{\omega_{i,j+1}}_d$, and of the codimension $0$ cells
         corresponding to $(\underbrace{1,\ldots, 1}_{i+j+3})$ and
        to $(\underbrace{1,\ldots, 1}_{i+j+1})$. 
        
        Let us traverse the loop $\kappa_\om$, see the left diagram in \ref{fig3}.   
        When $\kappa_\om$ enters the cell, labelled  by 
        $(\underbrace{1,\ldots, 1}_{i+j+3})$, from 
        a point on $\mathring{\sR}^{\omega_{i,j+1}}_d$ the 
        $(i+1)$\textsuperscript{st} largest real root  splits into $2$ distinct 
        real roots. On its way towards  $\mathring{\sR}^{\omega_{i+1,j}}_d$,  
        the smaller of these two roots approaches the 
        $(i+1)$\textsuperscript{st} largest real root and at the end  merges with 
        it. Then, when entering the cell corresponding to 
        $(\underbrace{1,\ldots, 1}_{i+j+1})$,
        the real double root splits into two conjugate complex roots
        with small imaginary part and with real part between the 
        $(i+1)$\textsuperscript{st} and $(i+2)$\textsuperscript{nd} largest real 
        roots.  
        While traversing $\kappa_\om$, the value of the real part 
        passes the $(i+1)$\textsuperscript{st} real root and then, 
        when the imaginary parts vanish, $\kappa_\om$ is back in 
        $\mathring{\sR}^\omega_{i,j+1}$.
        Thus $\kappa_\om^\bullet$ corresponds to the relation
        $\gamma_{i,j+1} \, \gamma_{i+1,j}^{-1}$.\smallskip
        
    \item[\bf(22)] For $\om = (\underbrace{1,\cdots, 1}_i, 2, \underbrace{1,\cdots , 1}_j , 2, \underbrace{1 \cdots 1}_\ell ) \in \bfc\Theta_{=2}$, the codimension $2$ cell 
        $\mathring{\sR}^\om_d$ lies only in the boundary of
        the codimension $1$ cells $\mathring{\sR}^{\omega_{i+2+j,\ell}}_d$,
        $\mathring{\sR}^{\omega_{i,j+\ell+2}}_d$,
        $\mathring{\sR}^{\omega_{i+j,\ell}}_d$ and
        $\mathring{\sR}^{\omega_{i,j+\ell}}_d$.
        (Note that the last two coincide when $j = 0$).  
        It also lies in the boundary of the codimension $0$-cells
        $(\underbrace{1,\ldots, 1}_k)$ for $k = i+j+\ell,\,i+j+\ell+2,\, i+j+\ell+4$.
        
        Again, let us traverse the loop $\kappa_\om$, see the middle and the right  diagrams in \ref{fig3}.  
        We start in $\mathring{\sR}_d^{\omega_{i+j+2,\ell}}$ and, when entering the cell  labelled by $(\underbrace{1,\ldots, 1}_{i+j+\ell+4})$,
        we split the $(i+j+2)$\textsuperscript{nd} root into two distinct real roots.
        Then, on the way through 
        this cell, the $(i+1)$\textsuperscript{st} and $(i+2)$\textsuperscript{nd}
        largest real roots approach each other and finally merge at  
        a point belonging to  the wall $\mathring{\sR}_d^{\omega_{i,j+\ell+2}}$.
        After passing $\mathring{\sR}_d^{\omega_{i,j+\ell+2}}$,
        the $(i+1)$\textsuperscript{st} largest real root splits into
        two complex conjugate roots. Next, we move towards 
        $\mathring{\sR}_d^{\omega_{i,j+\ell}}$ by bringing the $i$\textsuperscript{th} and $(i+1)$\textsuperscript{st} largest real roots together. 
        Then we split this double root into two complex conjugate roots
        inside the $d$-cell  labelled by $(\underbrace{1,\ldots,1}_{i+j+\ell})$; further 
        we move the real part of these roots to a value between the $(i+j)$\textsuperscript{th}  and the $(i+j+1)$\textsuperscript{st} real roots, while
        letting the imaginary part to approach $0$. This leads us to
        $\mathring{\sR}^{\omega_{i+j,\ell}}$, from where we follow the
        analogous route back to $\mathring{\sR}_d^{\omega_{i+j+2,\ell}}$.
        Note that along the way, we made certain choices that determine which codimension
        $1$ cell we want to approach next. The arguments, preceding the
        lemma, show that all such choices lead to homotopic paths/loops. 

        As a consequence, the loop $\kappa_\om^\bullet$ translates into the
        relation $$\gamma_{i+j,\ell}\, \gamma_{i+j+2,\ell}\,\gamma_{i,j+\ell+2}^{-1}\,
        \gamma_{i,j+\ell+2}^{-1} = 1.$$  

  \end{itemize}

  Next we show that the relations (3) and (22) generate the kernel of the homomorphism \ref{eq:inc}.
  For each $\omega$ with $|\omega|' = 2$, consider a closed regular 
  neighborhood $U_\omega$ of the cell $\mathring{\sR}_d^\omega$ in the space 
  $\cP_d^{\mathbf {c}\paOm{\langle d]}{\geq 2}}\, \cup\, 
  \mathring{\sR}_d^\omega$. 
  We may assume that, for distinct $\omega$'s, 
  these neighborhoods are disjoint. The space $U_\omega$ fibers over 
  $\mathring{\sR}_d^\omega$ with  fibers being  $2$-disks. Since the base 
  $\mathring{\sR}_d^\omega$ is contractible, the fibration 
  $p_\om: U_\omega \to \mathring{\sR}_d^\omega$ is trivial.

  Set $X := \cP_d^{\mathbf {c}\paOm{\langle d]}{\geq 2}}$. 
  Adding  to $X$ a cell 
  $\mathring{\sR}_d^\omega$, where $\om \in \bfc\Theta_{=2}$ produces a new space $Y$. Its homotopy type is 
  the result of attaching a $2$-handle to $X$ along its boundary.  The spaces 
  $X$, $U_\om$, and $X\cap U_\om$ are path-connected. Thus, by the 
  Seifert--van Kampen Theorem, 
  $$\pi_1(Y) \approx \pi_1(X)\ast_{\pi_1(X \cap\, U_\om)} \pi_1(U_\om),$$ which is 
  the free product $\ast$ of the groups $\pi_1(X)$ and $\pi_1(U_\om)$, 
  amalgamated over $\pi_1(X \cap U_\om)$. 
  Since $U_\om$ is homotopy equivalent to the cell $\mathring{\sR}_d^\omega$, 
  we get $\pi_1(U_\om) = 0$. At the same time, using the triviality of $p_\om$,
  $X \cap U_\om = U_\om \setminus \mathring{\sR}_d^\omega$ is 
  homotopy equivalent to a loop $\kappa_\om$, which bounds a small  $2$-disk 
  normal to the stratum $\mathring{\sR}_d^\omega$. Therefore, 
  $\pi_1(Y) \approx \pi_1(X)/[\kappa^\bullet_\om]$. 
  Thus $\pi_1(Y)$ arises from $\pi_1(X)$ by factoring modulo one of 
  the relation of type (3) or (22).

  Since the $U_\om$'s were chosen mutually disjoint, 
  it follows that we can repeat this argument for all
  $\om \in \bfc\Theta_{=2}$ and conclude 
  that $\pi_1(\cP_d^{\bfc\Theta})$ arises as a quotient of 
  $\pi_1(\cP_d^{\mathbf {c}\paOm{\langle d]}{\geq 2}})$ by
  the relations (3) and (22) for $\om \in \bfc\Theta_{=2}$.
\end{proof}

The next theorem shows that in many cases the relations from 
\ref{lem2.14} yield a presentation of a \emph{free} group.

\begin{theorem} \label{main}
  Let $\Theta \subset \paOm{\langle d]}{\geq 2}$ be a closed
  poset such that either 
  \begin{itemize}
    \item[(i)] $\Theta$ contains all $\omega \in \bfOm_{\langle d]}$ such
      that $\om = (\underbrace{1\cdots 1}_i 2 \underbrace{1\cdots1}_j 2 \underbrace{1\cdots 1}_\ell)$ for some $j > 0$;
      
      or 
    \smallskip  
    \item[(ii)] $\Theta$ is such that for $\omega \in \bfOm_{\langle d]}$ with
       $|\om|' = 2$ we have 
       $$\om \in \Theta\; \Leftrightarrow \; \om = (\underbrace{1\cdots 1}_i 3 \underbrace{1\cdots1}_j)$$ for some $i,j \geq 0$.   
  \end{itemize}

  Then the fundamental 
  group $\pi_1(\cP_d^{\bfc\Theta})$ is a free group in $\rank(\bar H_{d-2}(\bar{\cP}_d^{\Theta}; \Z))$\footnote{In \cite{KSW}, we will compute $\rank(\bar H_{d-2}(\bar{\cP}_d^{\Theta}; \Z))$ in pure combinatorial terms.} generators.
  In particular, in case (ii), we have $\pi_1(\cP_d^{\bfc\Theta}) = \Z$ for
  $d \geq 4$.
\end{theorem} 
\begin{proof} ~

  \begin{itemize}
    \item[(i)]
      By \ref{lem2.14} and the assumptions on $\Theta$, 
      the fundamental group 
      $\pi_1(\cP_d^{\mathbf c \Theta})$ is a quotient of the free group
      by arbitrary relations of type (3) and some specific of type (22). 

      Any relation of type (3) identifies some $\gamma_{i,j+1}$ with 
      $\gamma_{i+1,j}$. 
      Any relation of type (22) that can appear identifies some 
      $\gamma_{i+j+2}$ with $\gamma_{i,j+\ell+2}$. 

      In either case the relation and one generator can be dropped
      without changing the isomorphism type of the group. 
      
      Note that the number of generators of the free group $\pi_1(\cP_d^{\bfc\Theta})$ is the rank of $H_1(\cP_d^{\bfc\Theta}; \Z)$. By the Alexander duality, it is equal the rank of $H^{d-2}(\bar{\cP}_d^{\Theta}; \Z)$.
      
      The assertion follows. 

\smallskip
      
    \item[(ii)]
      By \ref{lem2.14} and the assumptions on $\Theta$, 
      the fundamental group $\pi_1(\cP_d^{\mathbf c \Theta})$ is a quotient 
      of the free group by \emph{all} possible relations of type (22).

      First consider compositions of type 
      $(\underbrace{1,\ldots,1}_{i},2,2,\underbrace{1,\ldots,1}_{\ell})$
      of $d'=i+4+\ell \leq d$ where $d$ and $d'$ are of the same parity. 
      The corresponding relation $\gamma_{i,\ell}\,\gamma_{i+2,\ell}\,
      \gamma_{i,\ell+2}^{-1}\,\gamma_{i,\ell}^{-1} \hfill \break =1$ identifies 
      $\gamma_{i+2,\ell}$ and $\gamma_{i,\ell+2}$. As a consequence this 
       relation, one generator can be dropped. By iterating this
      procedure, we identify all $\gamma_{i+2,\ell}$ and $\gamma_{i,\ell+2}$ 
      for $i=0,\ldots, d'-4$ and even $d-d' \geq 0$. 
      As a consequence, for a fixed $d'$, we have identified 
      $\gamma_{i,\ell}$ and $\gamma_{i',\ell'}$, where $i+\ell = d'-2 = i'+\ell'$
      if both $i$ and $i'$ are even or both $i$ and $i'$ are odd.
      After this identification, the group is generated by 
      $\gamma_{1,d'-3}$ for $3 \leq d' \leq d$.  
 
      Let $d' \geq 5$ and
      consider $\omega = (2,1,2,\underbrace{1\cdots1}_{d'-5})$. Then 
      the relation 
      $$\gamma_{1,d'-5}\, \gamma_{3,d'-5}\,\gamma_{0,d'-2}^{-1}\,\gamma_{0,d'-4}^{-1} =1$$
      implies 
      $\gamma_{1,d'-5}\, \gamma_{0,d'-4}^{-1} = \gamma_{3,d'-5}\, \gamma_{0,d'-2}^{-1}$, which identifies  
      $\gamma_{1,d'-5}$ and $\gamma_{3,d'-5}$ and, by transitivity,
      $\gamma_{1,d'-5}$ and $\gamma_{1,d'-3}$.
      Making this identification and removing the corresponding relations 
      leaves a set of trivial relations on the unique remaining generator
      $\gamma_{1,1}$ for $d$ even or $\gamma_{1,2}$ for $d$ odd.
      Thus we have arrived at a group with one generator and no relation.
      The assertion now follows. 
  \end{itemize}
\end{proof}

As an immediate consequence, we obtain the following result about certain 
natural spaces of real polynomials. 

\begin{corollary} \label{cormain}
  The fundamental group of the following spaces is a free group.
  \begin{itemize} 
    \item[(i)] The space of all degree $d$ polynomials with no
        two distinct roots of multiplicity $\geq 2$.
    \item[(ii)] The space of all degree $d$ polynomials with no
        root of multiplicity $\geq 3$. In this case, the
        fundamental group is $\Z$.  
  \end{itemize} 
\end{corollary} 
\begin{proof}
   The assertions follow immediately from \ref{main}.
\end{proof}

\smallskip 
The following illuminating example shows however that, in general,  $\cP_d^{\mathbf{c}\Theta}$ is not  free.

\begin{example}\label{extorsion}
  Let $d = 6$ and consider the smallest closed $\Theta \subseteq \bfOm_{\langle 6]}$ containing 
  \begin{eqnarray}\label{counterexample} 
  \qquad \{(3,1), (1,3), (1,3,1,1),(1,1,3,1), (2,2,1,1),(1,2,2,1),(1,1,2,2),(2,1,1,2)\}.
  \end{eqnarray}
  
  By \ref{lem2.14}, we need to consider the group freely generated by
  $$\gamma_{1,1},\gamma_{2,0}, \gamma_{1,3},\gamma_{2,2},\gamma_{3,1},\gamma_{4,0}$$  as well as by the dummy generators $\gamma_{0,2} = \gamma_{0,4} = 1$. 
  We then have to impose the relations of type (3) for the elements
  $$
    (3,1,1,1),\,(1,1,1,3)$$
  and relations of type (22) for 
  $$(2,2),\, (2,1,2,1),\,(1,2,1,2).$$
  \ref{fig6} sketches the $2$-complex representing these relations.
  The type (3) relations identify $\gamma_{0,4}$ and $\gamma_{1,3}$ as well
  as $\gamma_{3,1}$ and $\gamma_{4,0}$. 
  The type (22) relation, corresponding to $(2,2)$, then identifies
  $\gamma_{0,2}$ and $\gamma_{2,0}$.
  After these identifications, the type (22) relations that correspond to 
  $(2,1,2,1)$ and $(1,2,1,2)$ can be transformed as follows:
  $$\gamma_{1,1}\, \gamma_{3,1}\, \gamma_{0,4}^{-1}\, \gamma_{0,2}^{-1} 
     \longrightarrow \gamma_{1,1}\, \gamma_{3,1},$$
  $$\gamma_{2,0}\, \gamma_{4,0}\, \gamma_{1,3}^{-1}\, \gamma_{1,1}^{-1} 
     \longrightarrow \gamma_{3,1}\, \gamma_{1,1}^{-1}.$$
  Thus we are left with the generators $\gamma_{1,1}$, $\gamma_{2,2}$ and $\gamma_{3,1}$,
  subject to $\gamma_{1,1}\, \gamma_{3,1} = 1 = \gamma_{3,1}\, \gamma_{1,1}^{-1}$. 
  After eliminating $\gamma_{3,1}$, we are left with the generators $\gamma_{1,1}$, $\gamma_{2,2}$
  subject to $\gamma_{1,1}^2=1$ which is a presentation of $\Z/2\Z$. 
  Hence, we have shown that $\pi_1(\cP_6^{\mathbf{c}\Theta}) \approx \Z \ast \Z/2\Z$ for $\Theta$ introduced in \ref{counterexample}.
  If one wants to get rid of the free factor $\Z$ as well, one can throw $(1,1,3,1)$ into $\Theta$ which yields an
  identification of $\gamma_{2,2}$ and $\gamma_{3,1} = 1$ and leaves the other relation untouched.
\end{example}

\begin{figure}
  \begin{picture}(0,0)%
     \includegraphics[width=0.7\textwidth]{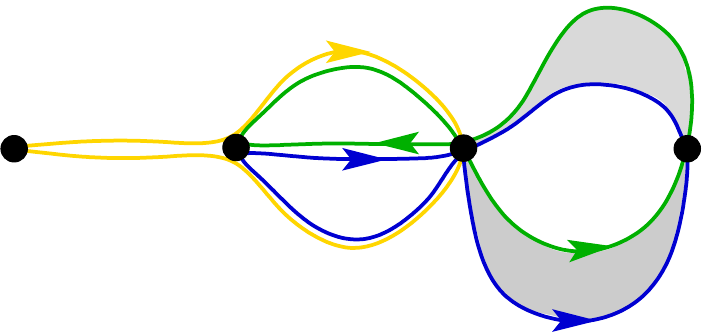}%
  \end{picture}%
  \setlength{\unitlength}{3947sp}%
  \begingroup\makeatletter\ifx\SetFigFont\undefined%
    \gdef\SetFigFont#1#2#3#4#5{%
    \reset@font\fontsize{#1}{#2pt}%
    \fontfamily{#3}\fontseries{#4}\fontshape{#5}%
    \selectfont}%
  \fi\endgroup%
  \begin{picture}(5769,2627)(1687,-3616)
    \put(5950,-1590){\makebox(0,0)[lb]{\smash{{\SetFigFont{10}{12}{\rmdefault}{\mddefault}{\updefault}{$\bf (3\,1\,1\,1)$}
     }}}}
    \put(5700,-3300){\makebox(0,0)[lb]{\smash{{\SetFigFont{10}{12}{\rmdefault}{\mddefault}{\updefault}{$\bf (1\,1\,1\,3)$}
     }}}}
    \put(2500,-2550){\makebox(0,0)[lb]{\smash{{\SetFigFont{10}{12}{\rmdefault}{\mddefault}{\updefault}{$\bf (2\,2)$}
     }}}}
    \put(4050,-1400){\makebox(0,0)[lb]{\smash{{\SetFigFont{10}{12}{\rmdefault}{\mddefault}{\updefault}{$\bf (2\,1\,1)$}
     }}}}
    \put(4050,-3200){\makebox(0,0)[lb]{\smash{{\SetFigFont{10}{12}{\rmdefault}{\mddefault}{\updefault}{$\bf (1\,1\,2)$}
     }}}}
    \put(4050,-2600){\makebox(0,0)[lb]{\smash{{\SetFigFont{10}{12}{\rmdefault}{\mddefault}{\updefault}{$\bf (1\,2\,1)$}
     }}}}
    \put(6600,-1200){\makebox(0,0)[lb]{\smash{{\SetFigFont{10}{12}{\rmdefault}{\mddefault}{\updefault}{$\bf (2\,1\,1\,1\,1)$}
     }}}}
    \put(5850,-2050){\makebox(0,0)[lb]{\smash{{\SetFigFont{10}{12}{\rmdefault}{\mddefault}{\updefault}{$\bf (1\,2\,1\,1\,1)$}
     }}}}
    \put(5800,-2750){\makebox(0,0)[lb]{\smash{{\SetFigFont{10}{12}{\rmdefault}{\mddefault}{\updefault}{$\bf (1\,1\,1\,2\,1)$}
     }}}}
    \put(6800,-3100){\makebox(0,0)[lb]{\smash{{\SetFigFont{10}{12}{\rmdefault}{\mddefault}{\updefault}{$\bf (1\,1\,1\,1\,2)$}
     }}}}
   \end{picture}%
  \caption{The assembly instructions for the $2$-dimensional $CW$-complex that realizes the fundamental group $\pi_1(\mathcal P_6^{\mathbf c\Theta})$ for $\Theta$ in \ref{counterexample}. The $2$-cells that realize the relations of type (3) are shaded. The three loops, to which the $2$-disks that realize relations of type (22) are attached, are shown in three colors (gold for $(2,2)$, green for $(2,1,2,1)$, blue for $(2,1,1,2)$).}
  \label{fig6}
\end{figure}

Using the same $\Theta$ as in \ref{counterexample} (\ref{extorsion}), but for $d = 8$, one can check 
that $\pi_1(\mathcal P_8^{\mathbf c\Theta})$ is a free group. 
Guided by this example, one is tempted to speculate that, for any 
closed $\Theta$ and $d$ large enough, 
$\pi_1(\cP_d^{\bfc\Theta})$ is free.

We will see that this is, in general,  not the case, but first we will show how
$\pi_1(\cP_d^{\bfc\Theta})$ stabilizes for $d$ large. This stabilization of the fundamental group follows almost immediately from the next lemma,  a
consequence of \ref{lem2.14}.

 \begin{lemma}
  \label{lemasym}
  Let $\Theta \subset \paOm{\langle d]}{\geq 2}$ be a closed poset
  and $\Theta' = \Theta \cap \bfOm_{\langle d-2]}$.
  Assume that no $\om$ with $|\om| \in \{d,d-2\}$ resides in 
  $\Theta_{=2}$.  
  Then $\pi_1(\cP_d^{\bfc\Theta}) \cong \pi_1(\cP_{d-2}^{\bfc\Theta'})$.
\end{lemma}
\begin{proof}
  By \ref{lem2.14} and the structure of $\Theta$,  all the relations of type 
  (3) and (22), that correspond to $\om \in \paOm{\langle d]}{=2}$ 
	with $|\om| \in\{d,d-2\}$ and define $\pi_1(\cP_{d}^{\bfc\Theta})$,
  are part of a presentation of $\pi_1(\cP_{d-2}^{\bfc\Theta'})$.
  Indeed, the generators $\gamma_{ij}$ for $i+j+2= d$ 
  and the relations of type (3) and (22), corresponding to 
  $\om \in \bfc \Theta_{=2}$ and $|\om| = d$,
  are the only difference to the presentations of 
  $\pi_1(\cP_{d-2}^{\bfc\Theta'})$ and $\pi_1(\cP_{d}^{\bfc\Theta})$.

  The relations of type (3) identify all $\gamma_{ij}$ for $i+j+2 \in \{d,d-2\}$.
  Since $\gamma_{0,d-2} = \gamma_{0,d-4} = 1$, this implies that 
  $\gamma_{ij} = 1$ for $i+j+2 \in \{d,d-2\}$. 
  The relations of type (22), corresponding to $\om \in \bfc\Theta_{=2}$
  with $|\om| = 2$, then turn into identifications of
  some $\gamma_{ij}$ for $i+j+2 = d-2$. But since these elements are already 
  identified with $1$, the relations of type (22) imply no independent identities, and
  hence can be removed.
  
  As a result, we have obtained the generators and relations for
  $\pi_1(\cP_{d-2}^{\bfc\Theta'})$, and the assertion follows.
\end{proof}

Now we can deduce a stabilization result for the fundamental group.

\begin{theorem} 
  \label{corasym}
  Let $\Theta \subseteq \paOm{\langle d]}{\geq 2}$ 
  be a closed poset. For $d' \geq d+2$ such that $d' = d \equiv 2$, 
  let $\Theta_{d'}$ be the 
  smallest closed poset in $\bfOm_{\langle {d'}]}$ containing
  $\Theta$. Then for $d' \geq d+2$, we have an isomorphism 
	$\pi_1(\cP_{d'}^{\bfc\Theta_{d'}}) \cong \pi_1(\cP_{d+2}^{\bfc\Theta_{d+2}})$.
\end{theorem} 
\begin{proof}
  We prove the assertion by induction on $d'$. 
  For $d' = d+2$, the assertion is trivial, since the relevant polynomial spaces are identical. For $d' \geq d+4$,  
  we have that $\Theta_{=2}$ contains no $\om$ such that 
  $|\om| \in \{d',d'-2\}$. This implies that
  $(\Theta_{d'})_{=2}$ contains no $\om$ such that 
	$|\om| \in \{d',d'-2\}$, and $\Theta_{d'-2} = \Theta_{d'} \cap \bfOm_{\langle d'-2]}$. 
  Hence by \ref{lemasym},
  $ \pi_1(\cP_{d'}^{\bfc\Theta_{d'}}) \cong \pi_1(\cP_{d'-2}^{\bfc\Theta_{d'-2}}) $.
  The assertion now follows from the induction hypothesis.
\end{proof}

Before we can provide a negative answer to 
the question whether 
$\pi_1(\cP_{d}^{\bfc\Theta})$ is free for large $d$, we need the following
simple lemma, which is an immediate consequence of \ref{lem2.14} and the
definition of the free product of finitely presented groups.

\begin{lemma}
  \label{lemafree}
  Let $\Theta \subset \paOm{\langle d]}{\geq 2}$
  be closed. Assume that for some $d' < d$ such that $d' \equiv d \mod 2$, 
  the poset $\Theta$ contains all
  $\om \in \paOm{\langle d]}{=2}$ with $|\om| = d'$. 

  Then $\pi_1(\cP_d^{\bfc\Theta})$ is the free product of
  $\pi_1(\cP_{d'-2}^{\bfc\Theta})$ and a group  presented by
  relations of type (3) and (22) for $\om \not\in \Theta$, such that 
  $|\om|> d'$ and $|\om|' = 2$.
\end{lemma}

Let $\Theta \subset  \paOm{\langle 6]}{=2}$ be the poset from \ref{extorsion} (see \ref{counterexample}).
For $d \geq 10$, let  $\Theta_d$ be the smallest 
closed poset in $\bfOm_{\langle d]}$ containing $\Theta$ and
all $\om \in \paOm{\langle d]}{=2}$ with $|\om| = 8$.
For $d \geq 10$, we have that $\Theta_d$ satisfies the conditions of
\ref{lemafree} with $d' = 8$. 
Then, for $d \geq 10$, $\pi_1(\cP_d^{\bfc\Theta_d})$ is the free product of
$\pi_1(\cP_d^{\bfc\Theta_6}) \cong \Z/2\Z$ and some group.
In particular, $\pi_1(\cP_d^{\bfc\Theta_d})$ is not free \emph{stably}, as $d \to \infty$.
\smallskip
 
Next we formulate some consequences of our results about spaces of 
polynomials with restrictions on the multiplicities of their \emph{critical points}. 
For any closed poset $\Theta \subset \bfOm$ and any $d > 0$, we denote by 
$\cP_{d+1}^{\crit \; \bfc \Theta}$ the 
space of polynomials of degree $d+1$, whose derivatives belong to 
$\cP_d^{\bfc \Theta}$.
The homotopy statement of the following corollary to \ref{lm2b} and \ref{main}
is well known.

\begin{corollary}
  \label{cor2.13} 
  For any closed poset $\Theta \subset \bfOm$ and any $d > 0$, the space 
  $\cP_{d+1}^{\crit\; \bfc \Theta}$ 
  is homotopy equivalent to the space $\cP_d^{\bfc\Theta}$. 
  In particular, $\pi_1(\cP_{d+1}^{\crit\; \bfc \Theta})$ is free for
  \begin{itemize}
    \item[(i)] 
      $\Theta = \paOm{\langle d]}{\geq 2}$,
\item[(ii)] $\Theta$ contains all $\omega \in \bfOm_{\langle d]}$ such
      that $\om = (\underbrace{1\cdots 1}_i 2 \underbrace{1\cdots1}_j 2 \underbrace{1\cdots 1}_\ell)$ for some $j > 0$.
    \item[(iii)] $\Theta$ is such that for $\omega \in \bfOm_{\langle d]}$ with
       $|\om|' = 2$ we have
       $$\om \in \Theta \Leftrightarrow \om = (\underbrace{1\cdots 1}_i 3 \underbrace{1\cdots1}_j)$$ for some $i,j \geq 0$.
  \end{itemize}
  Moreover, in case {\rm (iii)}, we have $\pi_1(\cP_{d+1}^{\crit\; \bfc \Theta}) = \Z$. 
\end{corollary}

\begin{proof} 
  The homotopy equivalence follows from the simple fact that 
  $\frac{d}{d x}: \cP_{d+1}^{\crit\; \bfc \Theta} \to \cP_d^{\bfc \Theta}$ 
  is a trivial fibration with fiber $\R$.
  Then (i) follows from \ref{lm2b} and (ii)-(iii) from \ref{main}.
\end{proof}

\ref{cor2.13} (iii) can be reformulated so that it can be seen as an analogue  
 of Arnold's \ref{th:Arnold1} for polynomials (instead of smooth functions).

\begin{corollary}
  \label{cor:arnA} 
  The fundamental group of the space of real monic polynomials of a fixed odd 
  degree $d > 1$ with no real critical points of multiplicity higher than $2$ 
  is isomorphic to the group $\Z$. 
\end{corollary}

\section{$\pi_1(\cP_d^{\mathbf {c}\paOm{\langle d]}{\geq 2}})$ and 
cobordisms of plane curves with restricted vertical tangencies}  
\label{sec:cobordism} 

Results and constructions in this section are similar to the ones from  
Arnold's \ref{th:Arnold2}.  They  also provide a glimpse into  a new area  of research whose goal is 
to describe and compute new ``bordism'' and ``quasi-isotopy'' theories of traversing flows on manifolds 
with boundary (\cite{Ka4}, \cite{Ka5}). The flows under investigation have constrained tangency to 
the boundary patterns. Crudely, one might think of such theories as classical smooth bordisms groups 
$\mathcal B_\ast(\cP_d^{\bfc\Theta})$ of the spaces $\cP_d^{\bfc\Theta}$. 
The papers \cite{Ka4} and \cite{Ka5} depend crucially on the present article and its sequel \cite{KSW} and 
contain the multidimensional generalizations of some results from this section. 
They deal with immersions $\b: X \to \R \times Y$ of $n$-dimensional compact smooth manifolds $X$ into the 
product $\R \times Y$, where the compact $n$-dimensional manifold $Y$ is fixed, and the $\b$'s are considered up to bordisms 
in the realm of immersions. All the immersions involved are required to avoid a priory chosen tangency patterns 
$\Theta$ to the fibers of the projection map $\R \times Y \to Y$. In \cite{Ka5}, these considerations and computations 
are applied to study the, so called, {\it convex quasi-envelops of traversing flows} and their bordisms. 

\smallskip

\begin{figure}
  \centering
  \includegraphics[width=0.8\textwidth]{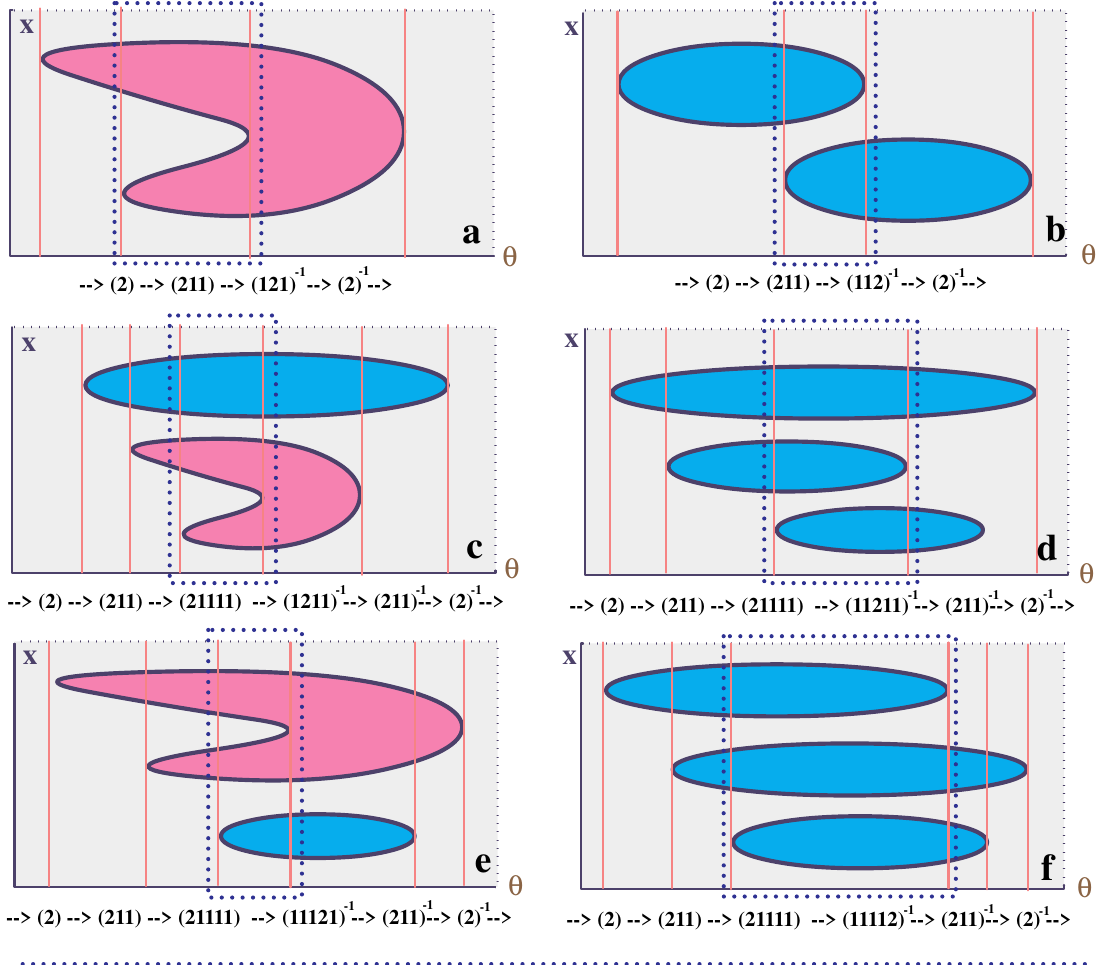}
   \caption{\small{A set $\{a, b, c, d, e, f\}$ of six generators, freely 
	generating the bordism group 
	$\mathcal B\big(\R \times S^1;\; \bfc\paOm{\langle 6]}{\geq 2}\big) 
	\approx \pi_1\big(\cP_6^{\bfc\paOm{\langle 6]}{\geq 2}}\big)$, is 
	shown as collections of regularly embedded curves in the cylinder with the coordinates 
	$(x, \psi) \in \R \times S^1$. Each collection of curves is generated 
	by a specific map 
	$\gamma: S^1 \to \cP_6^{\mathbf {c}\paOm{\langle 6]}{\geq 2}}$ as the 
	set of pairs $(\psi, x)$ with the property  $\gamma(\psi)(x) = 0$. 
	Each line $\{\psi = const\}$ is either transversal to the collection of curves, 
	or is quadratically tangent to it. No double tangent lines are 
	permitted. Each collection of curves is equipped with the circular 
	word (written under each of the six diagrams) that reflects the 
	transversal intersections of the loop $\gamma(S^1)$ with the 
	discriminant variety $\mathcal D_6 \subset \cP_6$.}} 
   \label{fig:kidney}
\end{figure}

The main result of this section is \ref{cobordism} whose proof is based on a number of technical results stated below.

\subsection{On classifiying maps to $\cP_d^{\bfc\Theta}$ }
Consider a compact smooth $n$-manifold $Y$ and an \emph{immersion} 
$\b: X \to \R \times Y$ of a smooth closed $n$-manifold $X$ into the interior 
of $\R \times Y$.  We denote by $\mathcal L$ the one-dimensional foliation, 
defined by the fibers $\mathcal L_y = \pi^{-1}(y)$ of the projection map $\pi: \R \times Y \to Y$. For 
each point $\sfx \in X$, we define 
$\mu_\b(\sfx)$ as the multiplicity of tangency between the 
$\sfx$-labeled branch $\b(X)_\sfx$ of $\b(X)$ -- the $\b$-image of the vicinity 
of $\sfx \in X$ -- and the leaf of $\mathcal L$ through 
$\b(\sfx)$. In general, the multiplicity $\mu_\b(\sfx)$ is either a natural number or $+\infty$.
However, in our settings, it is assumed to be finite. If the branch is transversal to the leaf, 
then $\mu_\b(\sfx) =1$.

We fix a natural number $d$ and assume that $\b$ is such that each leaf 
$\ell_\sfy$ of $\mathcal L$, $\sfy \in Y$, hits $\b(X)$ so that the 
following inequality holds:
\begin{eqnarray}\label{eq2.0}
  m_\b(\sfy) := \sum_{\{a \in \ell_\sfy \cap \b(X)\}} 
  \Big(\sum_{\{\sfx \in \b^{-1}(a)\}} \mu_\b(\sfx)\Big) \; \leq \; d.
\end{eqnarray}

We order the points $\{a_i\}$ of $\ell_\sfy \cap \b(X)$ by the 
values of their projections to $\R$ and introduce the combinatorial pattern 
$\om^\b(\sfy)$ of $\sfy \in Y$ as the sequence of multiplicities 
$\big\{\omega_i(\sfy) := \sum_{\{\sfx \in \b^{-1}(a_i)\}} \,
\mu_\b(\sfx)\big\}_i$. We denote by $D^\b(\sfy)$ the real divisor 
of the intersection $\ell_\sfy \cap \b(X)$, taken with 
multiplicities $\{\omega_i(\sfy)\}_i$.

\begin{proposition}\label{LIFT} 
  Let $Y$ be a smooth compact $n$-manifolds and $X$ a smooth closed $n$-manifold. 
  Then for any immersion $\b: X \to \R \times Y$, that satisfies \ref{eq2.0} together with the parity condition 
  $m_\b(\sfy) \equiv d \mod 2$ for all $y \in Y$, there exists a continuous map $\Phi_\b: Y \to \cP_d$ such that 
  $$\big\{(x, \sfy) \in  \R \times Y \big|\; \Phi_\b(\sfy)(x) = 0\big\} = 
	\b(X).$$

  If, for a given closed poset $\Theta \subset \bfOm_{\langle d]}$, 
  the immersion $\b$ is such that no $\omega^\b(\sfy)$ belongs to 
  $\Theta$, then $\Phi_\b$ maps $Y$ to $\cP_d^{\bfc\Theta}$.
\end{proposition}

\begin{proof} 
  The following claim implies the assertion:
  There are smooth functions $\{a_j: Y \to \R\}_j$ 
  such that $\b(X)$ is the union of the solution sets of the 
  equations $\big\{\,x^d + \sum_{j=0}^{d-1}a_j(\sfy)\,x^j = 0\,\big\}_{\sfy \in Y}$.

  Let us justify this claim. By \cite[Lemma 4.1]{Ka2} and 
  Morin's Theorem \cite{Mor1, Mor2}, if a particular branch 
  $\b(X)_\kappa$ of $\b(X)$ is tangent to the leaf $\ell_{\sfy_0}$ 
  at a point $b = (\a, \sfy_0)$ with the order of tangency $j = \mu_{\b, \kappa}(b)$, then 
	there is a system of local coordinates $(u, \tilde y, \tilde z) \in \R \times \R^ j \times \R^ {n-j}$ in the 
  vicinity of $b$ in $\R \times Y$ such that: 

  \begin{itemize}
    \item[(1)] 
      $\b(X)_\kappa$ is given by the equation 
		  $\big\{u^j + \sum_{k=0}^{j-2}\,\tilde y_k \,u^k = 0\big\}$;

    \item[(2)] each nearby leaf $\ell_{\sfy}$ is given by the 
	    equations $\{\tilde y = \overrightarrow{const},\, \tilde z = \overrightarrow{const'}\}$. 
  \end{itemize}
 
  Setting $u=x-\a$ and writing the $\tilde y_k$ as smooth functions of 
  $\sfy \in Y$, the same $\b(X)_\kappa$ can be given by the equation 
  $$\big\{P_{\a, \kappa}(x, \sfy) := 
  (x-\a)^j + \sum_{j=0}^{j-2}a_{\kappa, k}(\sfy)\,(x-\a)^k = 0\big\},$$ 
  where $a_{\kappa, k}: Y \to \R$ are smooth functions, vanishing at 
  $\sfy_0$. 
  Therefore, there exists an open neighborhood $U_{\sfy_0}$ of 
  $\sfy_0$ in $Y$ such that, in $\R \times U_{\sfy_0}$, 
  the locus $\b(X)$ is given by the monic polynomial equation  
  $$\Big\{P_{\sfy_0}(x, \sfy) := 
    \prod_{(\a,\, \sfy_0)\, \in \, \ell_{\sfy_0} \cap\, \b(X)} 
    \Big(\prod_{\kappa \in A_\a} P_{\a, \kappa}(x, \sfy)\Big) = 0 \Big\},$$
  of degree $m_\b(\sfy_0) \leq d$ in $x$. Here the finite set $A_\a$ 
  labels the local branches of $\b(X)$ that contain the point 
  $(\a, \sfy_0) \in \ell_{\sfy_0} \cap \b(X)$. 

  By multiplying $P_{\sfy_0}(x, \sfy)$ with 
  $(x^2 +1)^{\frac{d- m_\b(\sfy_0)}{2}}$, we get a polynomial 
  $\tilde P_{\sfy_0}(x, \sfy)$ of degree $d$, which, for each 
  $\sfy \in U_{\sfy_0}$, shares with 
  $P_{\sfy_0}(x, \sfy)$ the zero set 
  $\b(X) \cap (\R \times U_{\sfy_0})$, as well as the divisors 
  $D^\b(\sfy)$.

  For each $\sfy \in Y$, consider the space $\mathcal X_\b(\sfy)$ of 
  monic polynomials $\tilde P(x)$ of degree $d$ such that their real divisors 
  coincide with the $\b$-induced divisor $D^\b(\sfy)$. We view 
  $\mathcal X_\b := \coprod_{\sfy \in Y} \mathcal X_\b(\sfy)$ as a 
  subspace of $Y \times \mathcal P_d$. It is equipped with the obvious 
  projection $p: \mathcal X_\b \to Y$. The smooth sections of $p$ are exactly 
  the smooth functions $\tilde P(x, \sfy)$ that are of interest to us. Each 
  $p$-fiber $\mathcal X_\b(\sfy)$ is a convex set. 
  It follows that, given several smooth 
  sections $\{\sigma_i\}_i$ of $p$, we conclude that 
  $\sum_i \phi_i \cdot \sigma_i$ 
  is again a section of $p$, provided that the smooth functions 
  $\phi_i: Y \to [0, 1]$ form a partition of unity. 

  Since $X$ is compact, $\pi(\b(X)) \subset Y$ is compact as well. Thus, it 
  admits a finite cover by the open sets $\{U_{\sfy_i}\}_i$ as above. 
  Let $\{\phi_i: Y \to [0, 1]\}_i$ be a smooth partition of unity, 
  subordinated to this finite cover. Then the monic $x$-polynomial 
  $$\tilde P(x,\sfy) := \sum_i \phi_i(\sfy)\cdot 
  \tilde P_{\sfy_i}(x, \sfy)$$ of degree $d$ has the desired 
  properties. In particular, its divisor is $D^\b(\sfy)$ for each 
  $\sfy \in Y$. Thus, using $\tilde P(x, \sfy)$, any immersion 
  $\b: X \to \R \times Y$, such that no $\omega^\b(\sfy)$ belongs to 
  $\Theta$, is realized by a smooth map $\Phi_\b: Y \to  \cP_d^{\bfc\Theta}$ 
  for which $\b(X) = \{\Phi_\b(\sfy)(x) = 0\big\}$. 
\end{proof}

\subsection{From the fundamental group to cobordisms of embeddings and back} 
We denote by $\mathcal L$ the foliation of the cylinder  
$S^1 \times \R$, formed by the fibers $\{\ell_\psi\}_{\psi \in S^1}$ 
of the obvious projection $q: \R \times S^1 \to S^1$, and by 
$\mathcal L^\bullet$ the $1$-dimensional foliation of $\R \times S^1 \times [0, 1]$, formed by the 
fibers of the obvious projection $Q: \R \times S^1 \times [0, 1] \to S^1 
\times [0, 1]$. We pick a base point $\psi_\star \in S^1$ and the leaf 
$\ell_\star$ of $\mathcal L$ that corresponds to $\psi_\star$. Similarly, 
for each $t \in [0,1]$, we fix the base leaf  $\ell_\star(t)$ of 
$\mathcal L^\bullet$ passing through the point $(\psi_\star, t)$ on the base.

\medskip
We consider  {\it regular embeddings} $\b: M \to S^1\times \R$ of a collection of disjoint circles $S^1$  denoted by $M$ such 
that:
\begin{itemize}
	\item[(P1)] for each $\psi \in S^1$, the multiplicity defined in  \ref{eq2.0} satisfies the inequality $m_\b(\psi) \leq d$; 
	\item[(P2)] no leaf $\ell_\psi$ of $\mathcal L$ has  the combinatorial tangency pattern $\om^\b(\psi)$ with $\beta(M)$ belonging to the poset $\paOm{\langle d]}{\geq 2}$; thus, $\om^\b(\psi) \in \paOm{\langle d]}{\leq 1} = \mathbf c\paOm{\langle d]}{\geq 2}$ so that the map $q\circ\b: M \to S^1$ has only Morse type singularities,
	\item[(P3)] $\b(M) \cap \ell_\star = \emptyset$.
\end{itemize}

(An embedding is called {\it regular} if it is also an immersion, i.e. each tangent space to a point in the source is mapped non-degenerately to the tangent space of the image point). 

\medskip
The next definition explains our notion of \emph{cobordism of regular embeddings}, which deviates from the standard  cobordism theory \cite{Ka4}. (The same definition works for  immersions). 

\begin{definition}\label{bordism}
 We say that regular embeddings $\b_0: M_0 \to S^1\times \R$, $\b_1: M_1 \to S^1\times \R$ of collections of circles $M_0$ and $M_1$ are 
\emph{cobordant}, if there exists a compact smooth orientable surface $W$ 
with boundary $\d W = M_1 \coprod (-M_0)$ and a regular embedding  
$B: W \to \R \times S^1 \times [0,1]$ such that:
\begin{itemize} 
  \item $B|_{M_0} = \b_0$ and $B|_{M_1} = \b_1$;
  \item the projection $W \stackrel{B}{\rightarrow} \R \times S^1 \times [0,1] \to [0,1]$ is a Morse function for which $0$ and $1$ are regular values;
  \item for each $(\psi, t) \in S^1 \times [0,1]$, the multiplicity $m_B((\psi, t)) \leq d$ (see \ref{eq2.0});
  \item for each $(\psi, t) \in S^1 \times [0,1]$, the tangency pattern $\omega^B((\psi, t))$ does not belong to $\paOm{\langle d]}{\geq 2}$;
\end{itemize}

We denote by $\mathcal B(\R \times S^1; \mathbf{c}\paOm{\langle d]}{\geq 2})$ the set of 
cobordism classes of such embeddings $\b: M \to S^1\times \R$. 
\end{definition}

Note that $\mathcal B(\R \times S^1; \mathbf {c}\paOm{\langle d]}{\geq 2})$ is the set of cobordism classes of \emph{regularly embedded} curves in $\R \times S^1$ (and not the usual group of  cobordisms of singular $1$-manifolds with the target space $\R \times S^1$). Note also that the locus $B^{-1}(\R \times S^1, t) \subset W$ may fail to satisfy the requirements (P1)-(P3)  for some $t \in (0, 1)$. In particular, $B^{-1}(\R \times S^1, t)$ may fail to be the image under a regular embedding of a $1$-dimensional manifold. However, the second bullet in \ref{bordism} insures  that $B^{-1}(\R \times S^1, t)$ considered as a bivariate function may only have the Morse-type singularities (i.e., maxima/minima or saddles).\smallskip

In fact, the set $\mathcal B(\R \times S^1;  \mathbf{c} \paOm{\langle d]}{\geq 2})$  carries a 
\emph{group} structure, where the group operation $\b \odot \b '$ is defined as follows. 
Since $\b(M) \cap \ell_\star = \emptyset$ and 
$\b'(M') \cap \ell_\star = \emptyset$, we may view $\b(M)$ as subset of the 
strip $(0, 2\pi) \times \R$, and $\b'(M')$ as subset of the strip 
$(2\pi, 4\pi) \times \R$. Concatenating we get $\b(M) \coprod \b'(M') \subset [0, 4\pi] 
\times \R$. Rescaling $\lambda: [0, 4\pi] \to [0, 2\pi]$ we  
place the locus $\b(M) \coprod \b'(M')$ in back $[0, 2\pi] \times \R$, and 
thus in $S^1\times \R$. Evidently, this operation produces a pattern 
$\b(M) \odot \b'(M')$ satisfying (P1)-(P3).

\medskip
Consider the domain  
  $$\mathcal E_d := \big\{(x, \vec a) |\; P(x, \vec a) \leq 0\big\} \subset \R \times \cP_d,$$ where 
  $P(x, \vec a) := x^d + \sum_{j=0}^{d-1} a_j x^j$. We denote by $\d \mathcal E_d$ the boundary of $\mathcal E_d$. One can check that $\d \mathcal E_d$ is 
  a smooth hypersurface diffeomorphic to $\R^d$. 
  
  Let $\pi: \R \times \cP_d  \to  \cP_d$ denote the obvious projection. 
  Then $\pi^{-1}(\vec a) \cap \d \mathcal E_d$ is the support of the real 
  divisor $D_\R(P)$ of the $x$-polynomial $P(x, \vec a)$.\smallskip
  
\begin{definition}\label{def.E-reg} A smooth map $\Phi: Y \to \mathcal P_d$ is called $(\d\mathcal E_d)$-\emph{regular}
  if the  map $\Phi \times \mathsf{id} : Y \times \R \to \mathcal P_d \times \R$ is transversal to $\d \mathcal E_d$.
  \end{definition}
  
\begin{lemma}\label{lem.transversality} Let $Y$ be a compact smooth manifold.  Then the $(\d\mathcal E_d)$-regular maps form an open and dense set in the space of all smooth maps.  
\end{lemma}

\begin{proof}
  A smooth map $\Phi: Y \to \mathcal P_d$, given by $d$ functions  $a_{d-1}(y), \ldots$, $a_1(y)$, $a_0(y)$ on $Y$ (called {\it coefficients}),  
  is $(\d\mathcal E_d)$-regular if and only if, in any local coordinates $\vec y = \{y_1, \dots , y_n\}$ on $Y$, the system 
    \begin{empheq}[left=\empheqlbrace]{align}
             x^d + a_{d-1}x^{d-1} + \ldots + a_1x + a_0 & =  0   \nonumber \\
      d\, x^{d-1} + (d-1) a_{d-1}x^{d-2} + \ldots + a_1 & =  0  \label{eq.6.7_E_reg} \\ 
      \Big\{\frac{\d a_{d-1}}{\d y_j} u^{d-1} + \ldots + \frac{\d a_{1}}{\d y_j}u + \frac{\d a_0}{\d y_j} & = 0\,\Big\}_{j \in [1, n]} \nonumber
    \end{empheq}
of $(n+2)$ equations has no solutions in $\vec y$ for all $x$, and a similar property holds for $\d Y$. Indeed, $\d\mathcal E_d$ is given by the equation $$\wp(x, \vec a) := x^d + a_{d-1}x^{d-1} + \ldots +a_1x + a_0 = 0.$$ The pull-back $\Psi^\ast(\wp)$ of $\wp$ under the map $\Psi = (\mathsf{id}, \Phi): \R \times Y \to \R \times \mathcal P_d$ is the function $$x^d + a_{d-1}(y) x^{d-1} + \ldots +a_1(y) x + a_0(y)$$ on $\R \times Y$. Thus, the first equation in \ref{eq.6.7_E_reg}  defines the preimage of $\d\mathcal E_d$ under $\Psi$.  The transversality of $\Psi$ to $\d\mathcal E_d$ can be expressed as the non-vanishing of the $1$-jet of $\Psi^\ast(\wp)$ along the locus $\{\Psi^\ast(\wp) = 0\}$. In local coordinates on $\R \times Y$, the vanishing of the jet $\mathsf{jet}^1(\Psi^\ast(\wp))$ is exactly the constraints given by \ref{eq.6.7_E_reg}.
\smallskip

Note that, for each $x \in \R$, the system \ref{eq.6.7_E_reg} imposes $(n+2)$ \emph{affine} constraints on the functions $\{a_k: Y \to \R\}_{k \in [1,d]}$ and their first derivatives $\big\{\frac{\d a_k}{\d y_j}\big\}$. Therefore, for any $x$, \ref{eq.6.7_E_reg} defines an affine subbundle $\mathcal W(u)$ of the jet bundle $\{\mathsf J^1(Y, \mathcal P_d) \to Y\}$. Thus, the union $\mathcal W = \bigcup_{u \in \R} \mathcal W(u)$ is a ruled variety, residing in $\mathsf J^1(Y, \mathcal P_d)$. Since $\text{codim}(\mathcal W(u)) = n+2$, the codimension of $\mathcal W$ in $\mathsf J^1(Y, \mathcal P_d)$ is $n+1$. 

Consider the jet map $\mathsf{jet}^1(\Phi): Y \to \mathsf J^1(Y, \mathcal P_d)$. By the Thom Transversality Theorem (see \cite{Th} or \cite{GG}, Theorem 4.13), the space of $\Phi$ for which $\mathsf{jet}^1(\Phi)$ is transversal to the subvariety $\mathcal W$ is open and dense (recall that $Y$ is compact). Since $Y$ is $n$-dimensional, this transversality implies that $(\mathsf{jet}^1(\Phi))(Y) \cap \mathcal W = \emptyset$ for an open and dense set of maps $\Phi$. 

Similar arguments apply to the smooth maps $\Phi^\d : \d Y \to \mathcal P_d$; so, we may first perturb  a given map $\Phi: Y \to \mathcal P_d$ to insure the $(\d\mathcal E_d)$-regularity of $\Phi^\d = \Phi|_{\d Y}$ and then perturb $\Phi$ to insure its $(\d\mathcal E_d)$-regularity, while keeping the regularity of $\Phi^\d$.   

Therefore, the set of $(\d\mathcal E_d)$-regular maps $\Phi$ is open and dense in $C^\infty(Y, \mathcal P_d)$. 
\end{proof}

\begin{corollary} \label{cor.transversality} Let $Y$ be a compact smooth manifold.  Consider a smooth map $\Phi: Y \to \cP_d^{\bfc\paOm{\langle d]}{\geq 2}}$  that is 
transversal to the non-singular part $\mathcal D_d^\circ := \mathcal D_d \cap \cP_d^{\bfc\paOm{\langle d]}{\geq 2}}$ of the discriminant variety $\mathcal D_d \subset \cP_d$.
 
 Then the map $\mathsf{id} \times \Phi : \R \times Y  \to \R \times \cP_d^{\bfc\paOm{\langle d]}{\geq 2}} $ is transversal to the hypersurface $\d \mathcal E_d$, i.e., $\Phi$ is $(\d\mathcal E_d)$-regular. Once more, we conclude that such maps $\Phi$ form an open and dense subset of the space $C^\infty(Y, \cP_d^{\bfc\paOm{\langle d]}{\geq 2}})$.
\end{corollary}

\begin{proof} For each $y \in Y$, consider the line $\ell_{y}:= \pi^{-1}(\Phi(y)) \subset  \R \times \cP_d^{\bfc\paOm{\langle d]}{\geq 2}}$ and a point $(P, x) \in \d \mathcal E_d \cap \ell_{y}$. If $x \in \R$ is a simple real 
  root of the polynomial $P$, then the line $\ell_{y}$ is transversal to the hypersurface 
  $\d \mathcal E_d$ at $(x, P)$. If $x$ is a real root of 
  multiplicity $2$, then $(x, P) \in \mathcal \R \times D_d^\circ$, and by 
  the transversality of $\Phi$ to $\mathcal D_d^\circ$, the map $\mathsf{id} \times \Phi$ is 
 transversal  to the boundary $\d\mathcal E_d$ at the point $(x, P)$. 
\end{proof}

The next result is similar to Arnold's \ref{th:Arnold2}, see Introduction. \ref{fig:kidney} illustrates case $d=6$, for which 
$\mathcal B(S^1\times \R; \mathbf {c} \paOm{\langle 6]}{\geq 2})$ is the free group on $6$ 
generators, represented by disjoint  loops.

\begin{theorem} \label{cobordism} 
 The cobordism group $\mathcal B(\R \times S^1; \mathbf{c} \paOm{\langle d]}{\geq 2}))$, where 
  $d \equiv 0 \mod 2$, is isomorphic to the fundamental group 
	$\pi_1(\cP_d^{\bfc\paOm{\langle d]}{\geq 2}}, pt)$, and thus is a free group in 
  $\frac{d(d-2)}{4}$ generators.
\end{theorem}

\begin{proof} 
Each continuous loop $\gamma: S^1 \to \cP_d$ produces a 
   locus $\Xi_\gamma$  
  in the cylinder $\R \times S^1$ given by the formula 
  $$\Xi_\gamma := \{(x, \psi) \in \R \times S^1\;|\; \gamma(\psi)(x)= 0\}.$$ 
  
Note that,  in general, $\Xi_\gamma$ is not an image of a $1$-dimensional compact manifold $M$ under an immersion or a regular embedding. This complication calls for an appropriate approximation to $\gamma$.
  
Since $\cP_d^{\bfc\paOm{\langle d]}{\geq 2}}$ is open in $\cP_d$, any loop $\gamma: S^1 \to \cP_d^{\bfc\paOm{\langle d]}{\geq 2}}$ can be approximated by a smooth loop $\tilde\gamma: S^1 \to \cP_d^{\bfc\paOm{\langle d]}{\geq 2}}$ that is in the homotopy class of $\g$ and  is 
transversal to the non-singular part $\mathcal D_d^\circ := \mathcal D_d \cap  \cP_d^{\bfc\paOm{\langle d]}{\geq 2}}$ of the discriminant variety $\mathcal D_d \subset \cP_d$. By  \ref{cor.transversality}, $\tilde\gamma$ is a $(\d\mathcal E_d)$-regular map.     
Thus, the locus $\Xi_{\tilde\gamma} \subset \R \times S^1$, being equal to $(\mathsf{id} \times \tilde\g)^{-1}(\d\mathcal E_d)$, is a $1$-dimensional smooth submanifold  of the cylinder $\R \times S^1$. By this construction, the  tangency patterns of $\Xi_{\tilde\gamma}$ to the leaves $\{\ell_\psi\}_{\psi \in S^1}$ belong to the "open" poset $\bfc\paOm{\langle d]}{\geq 2}$. Therefore, $\Xi_{\tilde\gamma} \subset \R \times S^1$ satisfies (P1)-(P2). If the image $\tilde\gamma(\star)$ of the base point $\star \in S^1$ belongs to the 
  $d$-cell $\mathsf R_d^{()} \subset \cP_d$ that represents polynomials 
  without real roots, then property (P3) is also satisfied. 

 In particular, the double tangencies to the leaves $\{\ell_\psi\}_\psi$  and the cubic 
	tangencies to $\{\ell_\psi\}_\psi$ are forbidden when $\tilde\gamma(S^1) \subset  \cP_d^{\mathbf{ c}\paOm{\langle d]}{\geq 2}  }$: 
	they correspond to certain compositions $\omega_\psi \in \paOm{\langle d]}{\geq 2}$. 
\smallskip

Note that, for each $\tilde\gamma: S^1 \to \cP_d^{\bfc\paOm{\langle d]}{\geq 2}}$ that is $(\d\mathcal E_d)$-regular,  applying the map $\Phi_{\tilde\b}$ from \ref{LIFT}  to the embedding $\tilde\b: \Xi_{\tilde\gamma} \hookrightarrow \R \times S^1$, we obtain the loop $\tilde\gamma$ back.  
 Therefore, the correspondence $\gamma \leadsto \tilde\g  \leadsto \Xi_{\tilde\gamma}$ is a good candidate for a realization of  the desired group isomorphism 
\begin{eqnarray}\label{eq.iso}
\Xi_\ast: \pi_1(\cP_d^{\mathbf{c} \paOm{\langle d]}{\geq 2}}, pt) \stackrel{\approx}{\rightarrow} \mathcal B(\R \times S^1; \mathbf {c} \paOm{\langle d]}{\geq 2}),
\end{eqnarray}
(where $pt \in \mathsf R_d^{()}$), which is a posteriori the inverse of the map $\Phi$ from \ref{LIFT}.\smallskip

We have  already shown that any regular embedding $\b: M \to \R \times S^1$ which satisfies (P1)-(P3) 
 produces a based loop $\gamma(\b): (S^1, \star) \to (\cP_d^{\mathbf {c}\paOm{\langle d]}{\geq 2}}, pt)$. 
 Thanks to \ref{LIFT}, the locus $\Xi(\gamma(\b)) \subset \R \times S^1$ produces the embedding $\b$. Therefore, $\Xi_\ast$ from \ref{eq.iso} is surjective.
 \smallskip
 
 It remains to show that:

  \begin{itemize}
	  \item[\sf{(1)}] homotopic $(\d\mathcal E_d)$-regular loops $\gamma_0, \gamma_1: S^1 \to \cP_d^{\mathbf{c} \paOm{\langle d]}{\geq 2}}$ produce cobordant 
	    patters $\Xi_{\gamma_0}$  $\Xi_{\gamma_1}$ in $\R \times S^1$ (so that the correspondence $\Xi_\ast$ in \ref{eq.iso} is well-defined); 

    \item[\sf{(2)}] if $\Xi(\gamma)  \subset \R \times S^1$ is cobordant in $\R\times S^1 \times [0, 1]$ to $\emptyset$ (in the sense of \ref{bordism}),  then $\gamma$ is contractible in 
	    $\cP_d^{\mathbf {c}\paOm{\langle d]}{\geq 2}}$ (i.e., $\Xi_\ast$ is an injective map). 
  \end{itemize}

Thanks to \ref{cor.transversality}, without lost of generality, we may assume that $\g_0, \g_1$ are $(\d\mathcal E_d)$-regular (equivalently, transversal to $\mathcal D_d^\circ$).
  By a general position argument, we may assume that the homotopy 
  $\Gamma: S^1 \times [0,1] \to \cP_d^{\bfc\paOm{\langle d]}{\geq 2}}$ 
  that links $\gamma_0$ and $\gamma_1$ is smooth and $(\d\mathcal E_d)$-regular.

  If $\Gamma$ is $(\d\mathcal E_d)$-regular, then   the map 
  $$\Lambda := \mathsf{id}_\R \times \Gamma:\; \R \times (S^1\times [0, 1])  \longrightarrow  \R \times \cP_d^{\bfc\paOm{\langle d]}{\geq 2}}$$ 
  is transversal to the hypersurface $\d \mathcal E_d \subset \R \times \cP_d$. 
 
  This transversality implies that $W := \Lambda^{-1}(\d\mathcal E_d)$ is a regularly embedded surface in the shell $(\R \times S^1 \times [0, 1]) \cong \R \times S^1 \times [0, 1]$.  
It delivers the desired cobordism between the loop patterns $W \cap (\R \times S^1 \times \{0\})$ and $W \cap (\R \times S^1 \times \{1\})$. As a result, the map $\Xi_\star$ is well-defined. \smallskip

 Let $\g: S^1 \to \cP_d^{\mathbf{c} \paOm{\langle d]}{\geq 2}}$ be a smooth $(\d\mathcal E_d)$-regular map such that the $1$-manifold $\Xi(\g)$ is the boundary of a smooth orientable surface $W \subset  \R \times S^1 \times [0, 1)$ as in \ref{bordism}. To validate {\sf(2)}, we  again  use \ref{LIFT} to produce a smooth 
  $(\psi, t)$-parameter family of $x$-polynomials $\{P(x, \psi, t)\}_{\psi \in S^1,\; t \in [0,1]}$, whose roots form the surface $W$. 
  
  \smallskip
   Using the second bullet from \ref{bordism}, and  \ref{LIFT}, we see that the $\psi$-family $\{P(x, \psi, t)\}_\psi$ 
  gives rise to a loop $\gamma_t: S^1 \to \cP_d^{\mathbf {c}\paOm{\langle d]}{\geq 2}}$ 
  that depends continuously on $t$. Since for $t_\star =1$, the $t_\star$-slice of $W$ is empty, the loop 
  $\gamma_{t=1}$ is a subset of $\mathsf R_d^{()}$ and thus is contractible in $\cP_d^{\mathbf {c}\paOm{\langle d]}{\geq 2}}$. Therefore, the loop $\gamma_0$ is contractible in 
  $\cP_d^{\mathbf {c}\paOm{\langle d]}{\geq 2}}$.\smallskip
  
  To verify that the map $\Xi_\ast$ from \ref{eq.iso} is a group homomorphism is trivial.  
  Finally, applying \ref{lm2b}, we get that  $\mathcal B(\R \times S^1; \mathbf{c} \paOm{\langle d]}{\geq 2})$ is a free group in $\frac{d(d-2)}{4}$ generators.
\end{proof}

Along the lines of \ref{bordism}, for any closed poset $\Theta \subset \paOm{\langle d]}{\geq 2}$, 
we can introduce the cobordism group $\mathcal B(\R \times S^1; d, \mathbf{c} \Theta)$ of regularly embedded closed $1$-manifolds in the cylinder $\R \times S^1$, embeddings that avoid the tangency patterns from $\Theta$ and the line $\ell_\star \subset \R \times S^1$, and such that the total multiplicity from \ref{eq2.0} is bounded from above by $d$. 

Since $\cP_d^{\bfc\Theta}$ is an open subset of $\cP_d$, \ref{lem.transversality} and \ref{LIFT} apply to maps $\Phi: S^1 \to \cP_d^{\bfc\Theta}$ and their homotopies. Note that some strata $\cP_d^\om$ of codimension $2$ may reside in $\cP_d^{\bfc\Theta}$. If we insist on making all the relevant maps of $S^1$ and their homotopies transversal to these strata $\cP_d^\om$ as well, then the 
$(\d\mathcal E_d)$-regularity of these maps will be insured. Therefore, repeating the proof of \ref{cobordism} word for word, 
we get the following claim.

\begin{theorem} For any closed poset $\Theta \subset \paOm{\langle d]}{\geq 2}$, the cobordism group $\mathcal B(\R \times S^1; d, \mathbf{c} \Theta)$ is isomorphic to the fundamental group $\pi_1(\cP_d^{\bfc\Theta}, pt)$.
\end{theorem}

\end{document}